\documentclass[preprint]{siamart251216}

\usepackage{hyperref}
\usepackage[numbers]{natbib}
\usepackage{url}
\usepackage{graphicx} 
\usepackage{amsmath}
\usepackage{mathtools}
\usepackage{changepage}
\usepackage[capitalize]{cleveref}
\usepackage{amssymb}
\usepackage{hyperref}
\usepackage{algorithm}
\usepackage{algorithmic}
\usepackage{enumitem}

\usepackage{enumitem}
\usepackage{wrapfig}
\usepackage[font=small]{caption}
\usepackage{subcaption}

\makeatletter
\renewcommand\p@subfigure{}   
\makeatother

\newcommand{\bR}{\mathbb{R}}
\newcommand{\bS}{\mathbb{S}}
\newcommand{\bP}{\mathbb{P}}
\newcommand{\sgmin}{\sigma_{\min}}

\newcommand{\diag}{\textup{diag}}

\newcommand{\wnabla}{\widetilde{\nabla}}
\newcommand{\Gf}{{G_f}}
\newcommand{\Gh}{{G_h}}
\newcommand{\wJ}{\widetilde{J}}
\newcommand{\cI}{\mathcal{I}}

\newcommand{\xmid}{x_{\textup{mid}}}

\newcommand{\UU}{U_t U_t^\top}

\newtheorem{defi}{Definition}
\newtheorem{assump}{Assumption}
\newcounter{rmk}
\renewcommand{\thermk}{\arabic{rmk}}

\newenvironment{rmk}[1][]%
{%
  \refstepcounter{rmk}%
  \par\medskip\noindent
  \textbf{Remark~\thermk.}%
  \if\relax\detokenize{#1}\relax\else\ \textbf{(#1)}\fi
  \ \normalfont
}%
{\par\medskip}

\usepackage{xcolor}

\usepackage{tcolorbox}

\newcommand{\lmin}{{\lambda_{\min}}}
\newcommand{\lmax}{\lambda_{\max}}

\DeclareMathOperator*{\argmin}{arg\,min}

\newenvironment{talign*}
  {\begingroup\let\displaystyle\textstyle
   \begin{equation*}\begin{aligned}}
  {\end{aligned}\end{equation*}\endgroup}
\allowdisplaybreaks

\title{Zeroth-Order Constrained Optimization from a Control Perspective via Feedback Linearization}
\headers{Zeroth-Order Constrained Optimization from a Control Perspective }{R. Zhang, G. Zardini, A. Ozdaglar, J. Shamma, N. Li}

\author{Runyu Zhang\thanks{Laboratory for Information and Decision Systems, Massachusetts Institute of Technology,  (\email{runyuzha@mit.edu}, \email{gzardini@mit.edu}, \email{asuman@mit.edu}).}
  \and Gioele Zardini\footnotemark[1]
  \and Asuman Ozdaglar\footnotemark[1]
\and Jeff Shamma\thanks{UIUC, Department of Industrial and Enterprise Systems Engineering
  (\email{jshamma@illinois.edu}}).
\and Na Li\thanks{Harvard University, School of Engineering and Applied Sciences, (\email{nali@seas.harvard.edu}})}

\begin{document}

\maketitle

\begin{abstract}
Safe derivative-free optimization under unknown constraints is a fundamental challenge in modern learning and control. Existing zeroth-order (ZO) methods typically still assume access to a first-order oracle of the constraint functions or restrict attention to convex settings, leaving nonconvex optimization with black-box constraints largely unexplored. We propose the zeroth-order feedback-linearization (ZOFL) algorithm for ZO constrained optimization that enforces feasibility without access to the first-order oracle of the constraint functions and applies to both equality and inequality constraints. The proposed approach relies only on noisy, sample-based gradient estimates obtained via two-point estimators, yet provably guarantees constraint satisfaction under mild regularity conditions. It adopts a control-theoretic perspective on ZO constrained optimization and leverages feedback linearization, a nonlinear control technique, to enforce feasibility. Finite-time bounds on constraint violation and asymptotic global convergence guarantees are established for the ZOFL algorithm. A midpoint discretization variant is further developed to improve feasibility without sacrificing optimality. Empirical results demonstrate that ZOFL consistently outperforms standard ZO baselines, achieving competitive objective values while maintaining feasibility.
\end{abstract}

\begin{keywords}
zeroth-order optimization; constrained optimization; feedback linearization; nonlinear control; feasibility guarantees
\end{keywords}

\begin{MSCcodes}
90C56, 93B52, 65K10
\end{MSCcodes}

\section{Introduction}
Designing safe learning methods is both important and challenging. Safety requires guarantees of feasibility at every step, which in turn demands reliable information about the system’s objectives and constraints. In many real-world settings, such information is only accessible through function evaluations—gradients are either unavailable, unreliable, or prohibitively expensive to compute. This makes derivative-free methods natural candidates: they update decisions from sampled outcomes without requiring gradient access. Yet, enforcing strict safety guarantees in these derivative-free settings remains largely unresolved.


Among derivative-free approaches, zeroth-order methods have attracted significant attention due to their simplicity and scalability to high dimensions~\cite{wang18stochastic,nesterov2017random}.
The core idea is to build stochastic \emph{gradient estimators} via finite differences of function evaluations and then plug them into standard gradient-based updates~\cite{berahas2022theoretical}.
For instance, two-point estimators perturb the decision along random isotropic directions (Gaussian or uniform on the unit sphere) and combine the function values to approximate gradients~\cite{nesterov2017random,tang_zeroth-order_2023,ren2023escaping}.
When combined with gradient descent, such estimators yield provable converge rates for unconstrained optimization.

In the constrained setting, most existing ZO algorithms assume black-box objectives but \emph{white-box constraints}.
Explicit knowledge of the constraint set enables efficient projections or local linearizations, ensuring feasibility. 
This has led to a variety of algorithms, including projection–gradient-descent~\cite{tang_zeroth-order_2023,liu_zeroth-order_2018,chen_zo-adamm_2019,ghadimi_mini-batch_2016}, Frank–Wolfe–type methods~\cite{sahu_decentralized_2020,liu24zeroth-order,balasubramanian_zeroth-order_2019}, and Sequential Quadratic Programming (SQP)-style approaches\\\cite{curtis_almost-sure_2023,berahas2025sequential}.\!
However, in many safe learning settings, such as safe RL or chance-constrained optimization~\citep{achiam2017constrained,cui2020adaptive,zhang2011chance}, the constraint functions themselves are \emph{unknown}. 
Here, only noisy zeroth-order estimates of the constraint gradients are available, and feasibility becomes much harder to enforce. 
Most existing work in this regime focuses on convex problems and relies on primal–dual schemes~\citep{zhou_zeroth-order_2025,chen_model-free_2022,yi_linear_2021,nguyen_stochastic_2022,li_zeroth-order_2022,maheshwari22,liu_min-max_2020,chen_continuous-time_2023,hu_gradient-free_2024}. 
For nonconvex problems, guarantees are scarce: some works provide only empirical evidence~\citep{zhou_zeroth-order_2025}, while others require solving expensive convex subproblems at each step~\citep{nguyen_stochastic_2022,li_zeroth-order_2022}. 
Feasibility in these settings remains fragile, typically degrading as noise in the gradient estimators increases~\citep{oztoprak_constrained_2021}.

In contrast, in the first-order (FO) setting, SQP methods are known to be highly effective for nonconvex constrained optimization.
They often outperform primal–dual approaches in practice, particularly when the number of constraints is small relative to the dimension of the decision variable~\citep{gould2003galahad,liu2018comparison,zhang2020proximal,applegate2024infeasibility}. 
Another complementary line of work interprets constrained optimization through a control-theoretic lens~\citep{cerone_new_2024,zhang2025constrainedoptimizationcontrolperspective}. 
Here, the optimization dynamics are viewed as a controlled system: the primal variables are states, the Lagrange multipliers are control inputs, and finding a first-order KKT point corresponds to steering the system to a feasible equilibrium. 
This perspective enables the use of nonlinear control tools, such as \emph{feedback linearization (FL)}, to design algorithms. 
Recent work shows that, under suitable conditions, FL-based schemes recover SQP-like updates and achieve strong performance on nonconvex problems~\citep{zhang2025constrainedoptimizationcontrolperspective}.

\vspace{3pt}
Despite their promise in first-order optimization, FL and SQP approaches have not been systematically studied in the zeroth-order regime. 
Extending them is nontrivial: FL and SQP depends critically on precise first-order information, but in the zeroth-order setting only noisy estimates are available, breaking the mechanisms that ensure feasibility. 
This raises the fundamental question:

\begin{adjustwidth}{0.7cm}{0.7cm}\emph{How can we design zeroth-order methods that handle nonconvex constrained optimization with only noisy gradient estimators, while still providing provable guarantees of constraint satisfaction?}
\end{adjustwidth}

\vspace{3pt}
\noindent\textbf{Our Contributions.} 
We develop a zeroth-order constrained optimization framework that extends feedback linearization ideas to the derivative-free regime, inspired by control and dynamical systems perspective~\citep{cerone_new_2024,zhang2025constrainedoptimizationcontrolperspective}. 
First, we show how to construct an FL scheme tailored to dynamics evolving under noisy gradient information, in contrast to prior approaches that rely on convex relaxations or primal–dual surrogates. 
Second, we demonstrate that full Jacobians are unnecessary: it is enough to approximate a small set of Jacobian–vector products, which can be efficiently estimated via two-point zeroth-order queries. 
Third, we establish theoretical guarantees (\cref{thm:zeroth-order-eq} and~\cref{thm:zeroth-order-ineq}) showing that constraint violations contract toward zero with high probability, up to controllable approximation and discretization errors. 
Finally, we provide empirical evidence that our method consistently achieves stronger feasibility performance than standard baselines, while achieving competitive objective values.


\vspace{3pt}
\noindent\textbf{Notations.}
We use~$\nabla f(x)$ to denote the gradient of a scalar function~$f: \bR^n \to \bR$ evaluated at the point~$x\in \bR^n$ and use~$\nabla^2 f(x)$ to denote its corresponding Hessian matrix. 
We use~$J_h(x)$ to denote the Jacobian matrix of a function~$h: \bR^n\to\bR^m$ evaluated at~$x\in \bR^n$, i.e.~$[J_h(x)]_{i,j} = \frac{\partial h_i(x)}{\partial x_j}$, $i\in [m], j\in[n]$. 
Unless specified otherwise, we use $\|\cdot\|$ to denote the $L_2$ norm of matrices and vectors. We also denote $[x]_+:=\max(x,0)$ where $\max$ is taken entrywise for a vector $x$.


\section{Preliminaries}\label{sec:prelim}
We begin by introducing the constrained optimization setup and reviewing prior work from a control perspective, which motivates our approach.
We then highlight the challenges unique to the ZO setting.

We consider constrained optimization problems of the form 
\begin{equation}
\label{eq:optimization-eq-contraint}
\textstyle    \min_{x\in \bR^n}f(x) \quad \text{s.t. }h(x)=0,
\end{equation}
where~$f:\bR^n\to \bR$ is the objective and~$h:\bR^n\to \bR^m$ encodes equality constraints.
Here we assume that~$f,h$ are differentiable, and additional assumptions will be introduced where needed to support the analysis.
The first-order Karush-Kuhn-Tucker (KKT) conditions are
\begin{equation}
\label{eq:KKT}
\nabla f(x)+J_h(x)^\top \lambda =0, \quad h(x)=0.
\end{equation}
Here,~$J_h(x)$ denotes the Jacobian of~$h$ and~$\lambda \in \bR^m$ are the Lagrange multipliers.

While we begin by focusing on equality-constrained problems for clarity of exposition, our analysis also extends to problems with \textit{inequality }constraints, which will be studied in Section \ref{sec:ineq}.
\subsection{\!First-order Constrained Optimization: A Control Perspective}
\begin{wrapfigure}{r}{0.4\textwidth} 
    \vspace{-18pt}
    \centering
\includegraphics[width=0.9\linewidth]{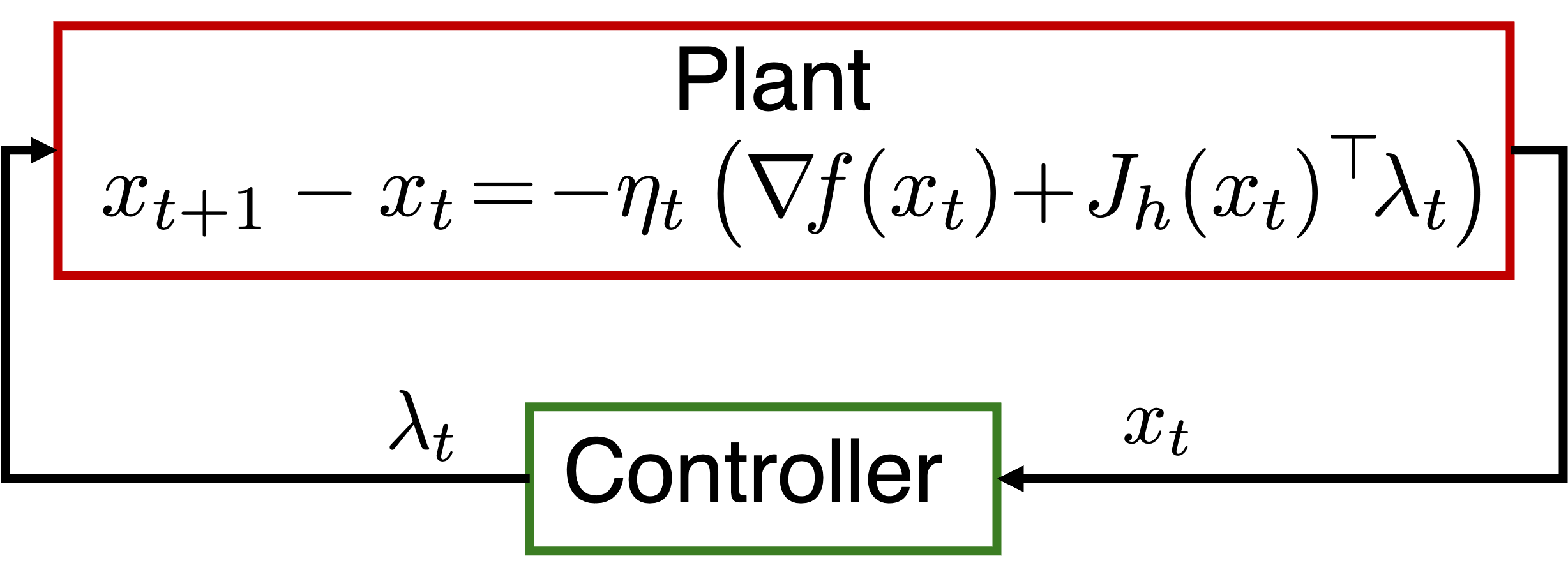}
\vspace{-5pt}
    \caption{\centering Control Perspective for constrained optimization.}
\label{fig:control-perspective-constrained-opt}
\vspace{-15pt}
\end{wrapfigure}

Recent works \cite{cerone_new_2024,zhang2025constrainedoptimizationcontrolperspective} interpret constrained optimization from a control perspective, offering new insights and enabling novel algorithmic designs in the first-order optimization regime. The key idea is to reinterpret \cref{eq:KKT} as the equilibrium of a dynamical system.
Specifically, define the updates
\begin{equation}\label{eq:FO-dynamics}
    x_{t+1} - x_t = -\eta_t\!\left(\nabla f(x_t) + J_h(x_t)^\top \lambda_t\right)\!, 
    ~~~ y_t = h(x_t),
\end{equation}
where~$x_t$ is the system state,~$y_t$ the constraint output, $\lambda_t$ the control input, and $\eta_t>0$ is the stepsize. Note that \eqref{eq:FO-dynamics} incoorporates a wide range of optimization algorithms (cf. \cite{nocedal_numerical_2006,neal2011distributed,chambolle2011first})

At any equilibrium~$(x^\star,\lambda^\star)$ of \cref{eq:FO-dynamics}, we have~$\nabla f(x^\star) + J_h(x^\star)^\top \lambda^\star = 0$.
If, in addition,~$x^\star$ is feasible (i.e.,~$h(x^\star)=0$), then~$(x^\star,\lambda^\star)$ satisfies the KKT
conditions \cref{eq:KKT}.
Hence, the control objective is to design~$\lambda_t$ so that the
closed-loop dynamics drive~$y_t \to 0$ and stabilize~$x_t$ at a feasible equilibrium
(see \cref{fig:control-perspective-constrained-opt}).

To design the controller~$\lambda_t$ to reach a feasible equilibrium, we next introduce the feedback linearization (FL) approach, which is the main focus of this paper.


\noindent\textbf{Feedback linearization (FL).} ~
FL is a classical control technique for stabilizing nonlinear systems of the form
\vspace{-5pt}
\begin{align}
    x_{t+1} - x_t = -b(x_t) + A(x_t)\lambda_t,
    \label{eq:non-linear}
\end{align}
by introducing a new input that cancels the nonlinearities~\cite{isidori1985nonlinear,henson1997feedback}.
If~$G(x)$ is invertible, one can apply a feedback transformation by introducing a new (virtual) control input 
$u_t$  and redefining the original input as ~$\lambda_t = A(x_t)^{-1}(b(x_t) + u_t)$. Substituting this expression into \eqref{eq:non-linear} yields: $x_{t+1} - x_t = u_t$, a linear system for which standard stabilizing controllers are available.

Recall the dynamics in~\cref{eq:FO-dynamics}.
Writing out the constraint evolution gives
\begin{equation}\label{eq:y-dynamics}
    y_{t+1} - y_t 
    \approx J_h(x_t)(x_{t+1}-x_t) = -\underbrace{\eta_t J_h(x_t)\nabla f(x_t)}_{b(x_t)} 
    \underbrace{- \eta_t J_h(x_t)J_h(x_t)^\top}_{A(x_t)} \lambda_t,
\end{equation}
where the terms can be viewed as~$b(x_t)$ and~$A(x_t)\lambda_t$ in~\cref{eq:non-linear}.
Hence choosing
\begin{equation*}
   \lambda_t = -\big(J_h(x_t)J_h(x_t)^\top\big)^{-1}\!\left(J_h(x_t)\nabla f(x_t)+u_t\right),
\end{equation*}
cancels the nonlinear dependence and yields the linearized dynamics~$y_{t+1}-y_t \approx \eta_t u_t$. In this paper, we consider a specific linear controller $u_t = -Ky_t$ where $K\in\bR^{m\times m}$ is a positive definite matrix, thus we get $y_{t+1} - y_t\approx -Ky_t$ and hence the constraints converge exponentially to zero.
This design gives the \emph{\textbf{first-order feedback linearization (FO-FL)}} method:

\begin{tcolorbox}[title={\small FO-FL (Equality Constraints) \cite{cerone_new_2024,zhang2025constrainedoptimizationcontrolperspective}},width=\linewidth]
\vspace{-10pt}
\begin{equation}\label{eq:feedback-linearization}
\begin{split}
    x_{t+1} - x_t &= - \eta_t\left(\nabla f(x_t) + J_h(x_t)^\top \lambda_t\right),\\
    \lambda_t &= -\left(J_h(x_t)J_h(x_t)^\top\right)^{-1}\!\left(J_h(x_t)\nabla f(x_t)-K h(x_t)\right).
\end{split}
\end{equation}
\end{tcolorbox}

FO-FL has been shown to effectively handle nonlinear dynamics, making it well-suited for nonconvex constrained optimization \cite{cerone_new_2024,schropp_dynamical_2000,zhang2025constrainedoptimizationcontrolperspective}.

\subsection{Zeroth-order Constrained Optimization: Baseline and Challenges}
\noindent\textbf{Problem Setup.}~
In many learning and control problems (e.g., safe RL), the gradients of $f$ and $h$ are unavailable; \emph{one can only query their values $f(x)$ and $h(x)$ at selected points $x$, without access to $\nabla f(x)$ or $J_h(x)$}. Zeroth-order optimization aims to solve \cref{eq:optimization-eq-contraint} using only such queries.

\begin{wrapfigure}{r}{0.4\textwidth} 
    \centering
    \vspace{-20pt}
\includegraphics[width=0.9\linewidth]{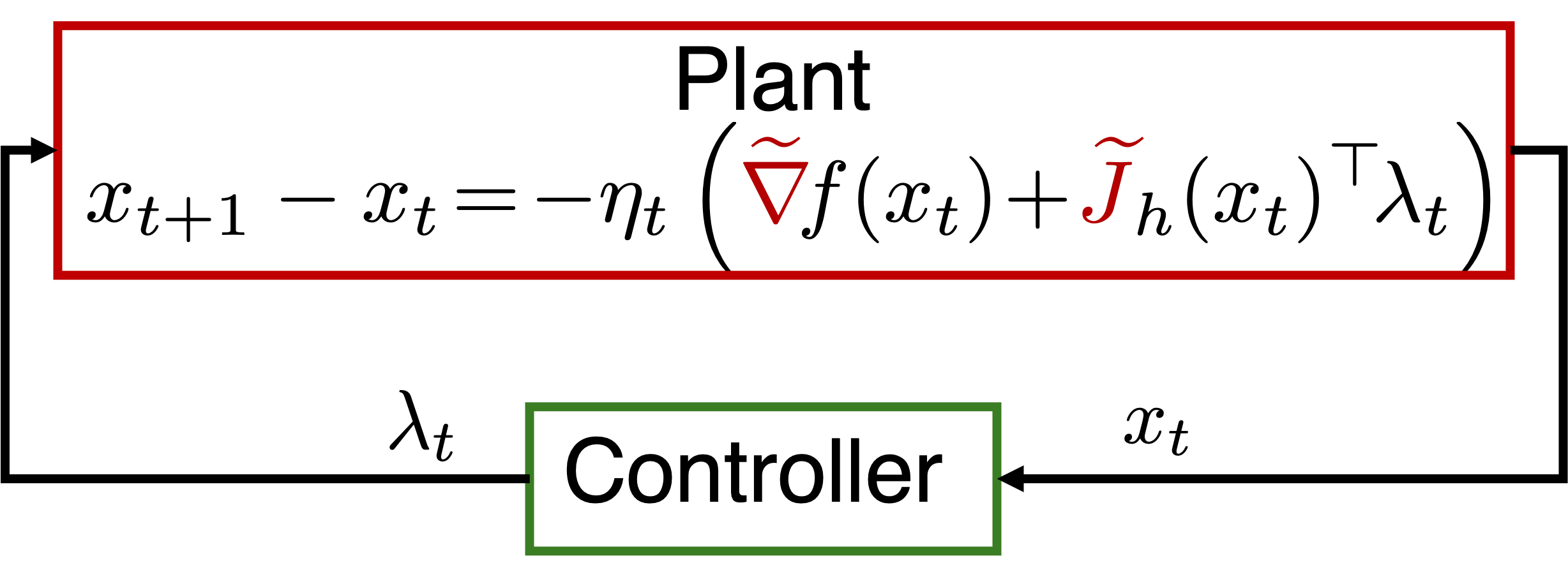}
\vspace{-5pt}
    \caption{\centering Control Perspective for \\Zeroth-Order Constrained Optimization}
\label{fig:control-perspective-constrained-opt-ZO}
\vspace{-30pt}
\end{wrapfigure}

The absence of first-order information motivates the use of stochastic finite-difference estimators for $\nabla f(x)$ and $J_h(x)$. A standard choice is the two-point estimator (cf. \cite{nesterov2017random,ren2023escaping,tang_distributed_2021}):
\begin{equation}\label{eq:two-point-estimator}
\begin{split}
    \wnabla f(x_t) &\!=\! \frac{n}{T_B}\sum_{i=1}^{T_B} 
    \frac{f(x_t \!+\! r_1u_i) \!-\! f(x_t \!-\! r_1u_i)}{2r_1}\,u_i,\!\!\!\!\!\!\!\!\!\!\!\!\!\!\!\!\!\!\!\!\!\!\!\!\\
    \wJ_h(x_t) &\!=\! \frac{n}{T_B}\sum_{i=1}^{T_B} 
    \frac{h(x_t \!+\! r_1u_i) \!-\! h(x_t \!-\! r_1u_i)}{2r_1}\,u_i^\top,\!\!\!\!\!\!\!\!\!\!\!\!\!\!\!\!\!\!\!\!\!\!\!\!
\end{split}
\end{equation}
where~$u_i$ are drawn i.i.d.\ from the~$n$-dimensional unit sphere.
These estimators are nearly unbiased in expectation when the radius $r_1$ is sufficiently small, but can be very noisy, particularly in high dimensions.

\noindent\textbf{A Zeroth-Order Baseline and Its Limitation.}~
Given the gradient estimator (\cref{eq:two-point-estimator}), a natural idea is to substitute these estimates directly into the FO-FL updates (\cref{fig:control-perspective-constrained-opt-ZO}).  This yields the following \emph{\textbf{zeroth-order baseline (ZO-baseline)}}:

\begin{tcolorbox}[title={\small ZO-baseline for Equality-Constrained Optimization},width = \linewidth]
\vspace{-10pt}
\begin{equation}\label{eq:ZO-baseline}
\begin{split}
    x_{t+1} - x_t &= - \eta_t\!\left(\wnabla f(x_t) + \wJ_h(x_t)^\top \lambda_t\right),\\
    \lambda_t &= -\!\left(\wJ_h(x_t)\wJ_h(x_t)^\top\right)^{-1}\!\!\left(\wJ_h(x_t)\wnabla f(x_t)-K h(x_t)\right).
\end{split}
\end{equation}
\end{tcolorbox}

This approach has been explored in recent work on noisy or biased estimators~\cite{oztoprak_constrained_2021}.
However, it suffers from a critical drawback: constraint satisfaction is no longer guaranteed. 
To see this, note that the constraint dynamics become
\begin{small}
\begin{talign*}
    &\quad h(x_{t+1}) - h(x_t) 
    \approx J_h(x_t)(x_{t+1}-x_t) \\
    &= -\eta_t\!\left(J_h(x_t)\wnabla f(x_t) 
    - \underbrace{\big(J_h(x_t)\wJ_h(x_t)^\top\big)\big(\wJ_h(x_t)\wJ_h(x_t)^\top\big)^{-1}}_{\neq I}
    \big(\wJ_h(x_t)\wnabla f(x_t) - Kh(x_t)\big)\right).
\end{talign*}
\end{small}
The mismatch between~$J_h(x_t)$ and~$\wJ_h(x_t)$ breaks the exact cancellation property of FO-FL, so the update no longer simplifies to~$-K h(x_t)$. 
As a result, the iterates are not guaranteed to converge to the feasible set~$\{x:h(x)=0\}$.

This limitation motivates the central question of our work:
\begin{quote}
\emph{Can we design zeroth-order methods that enforce constraint satisfaction despite relying on noisy gradient estimates (\cref{eq:two-point-estimator})?}
\end{quote}
In the next section, we show that a refined FL-based design yields a positive answer.

\section{Feedback-Linearization-Inspired Zeroth-order Algorithm}\label{sec:zeroth-order}
From the previous section, we know that simply substituting noisy gradient estimates into the FO-FL scheme does not guarantee feasibility.
A more careful design is required. In this section, we will present our algorithm along with the design insight and the theoretical guarantees on the constraint satisfaction. 
\subsection{Algorithm}

\noindent\textbf{Key idea.}
FL works by introducing a change of input that transforms nonlinear dynamics into a linear system.
In the ZO setting, however, the dynamics evolve under a \emph{noisy} gradient descent process (\cref{fig:control-perspective-constrained-opt-ZO}), which prevents the direct use of FO-FL.
To recover feasibility, we must rederive the FL scheme for this setting.

\noindent\textbf{Constraint dynamics.}
Consider the evolution of the constraints:
\begin{align}\label{eq:h-dynamics-ZO}
    \textstyle h(x_{t+1}) \!-\! h(x_t) \!\approx\! J_h(x_t)(x_{t+1}\!-\!x_t) \!=\! -\eta_t\!\left(J_h(x_t)\wnabla f(x_t) \!-\! J_h(x_t)\wJ_h(x_t)^\top \lambda_t\right).
\end{align}
If we choose
\begin{align}\label{eq:lambda-t-derivation}
    \textstyle \lambda_t = -\!\left(J_h(x_t)\wJ_h(x_t)^{\top}\right)^{-1}\!\big(J_h(x_t)\wnabla f(x_t)-K h(x_t)\big),
\end{align}
then~\cref{eq:h-dynamics-ZO} simplifies to
\vspace{-20pt}
\begin{align*}
    h(x_{t+1}) - h(x_t) \approx -\eta_t K h(x_t),
\end{align*}
which guarantees exponential decay of constraint violations.

\noindent\textbf{Challenge.} 
\cref{eq:lambda-t-derivation} requires access to the exact Jacobian~$J_h(x_t)$, which is not available in the ZO regime. 
At first glance, this seems to present a fundamental obstacle.

\noindent\textbf{Insight.}
A closer examination reveals that full access to $J_h(x_t)$ is unnecessary: it suffices to compute the Jacobian--vector products
$J_h(x_t)\,\wnabla f(x_t)$ and $J_h(x_t)\,\wJ_h(x_t)^\top$.
Equivalently, one only needs directional derivatives of $h$ along the directions $\wnabla f(x_t)$ and the rows of $\wJ_h(x_t)$, rather than the full Jacobian.
Crucially, these Jacobian–vector products can be efficiently approximated using standard two-point estimators as follows, thereby rendering the scheme implementable in the ZO setting:
\begin{equation}\label{eq:def-Gf-Gh}
\begin{split}
        \Gf & = \|\wnabla f(x_t)\|\frac{\left(h(x_t + r_2 v_f) - h(x_t - r_2 v_f)\right)}{2r_2}, ~~\textup{where } v_f = \frac{\wnabla f(x_t)}{\|\wnabla f(x_t)\|}\\
        [\Gh]_{:,i} &\!=\!  \|\wnabla h_i(x_t)\|\frac{\left(h(x_t \!+\! r_2 v_{h,i}) \!-\! h(x_t \!-\! r_2 v_{h,i})\right)}{2r_2}, ~\textup{where } v_{h,i} = \frac{\wnabla h_i(x_t)}{\|\wnabla h_i(x_t)\|},
        \end{split}
    \end{equation}
Here~$\wnabla h_i(x_t)$ is the transpose of the $i$-th row of $\widetilde{J}_h(x_t)$, i.e. $\wnabla h_i(x_t) \!= \![\widetilde{J}_h(x_t)]_{i,:}^\top$. The normalization factors $\|\wnabla f(x_t)\|$ and $\|\wnabla h_i(x_t)\|$ are introduced to convert directional finite differences along unit vectors into Jacobian--vector products along the original directions.
Indeed, since $v_f$ and $v_{h,i}$ are unit-norm directions, the centered difference $\frac{h(x_t + r_2 v) - h(x_t - r_2 v)}{2r_2}$ approximates $J_h(x_t)\,v$. Multiplying by the corresponding norm recovers
$J_h(x_t)\,\wnabla f(x_t)$ and $J_h(x_t)\,\wnabla h_i(x_t)$, respectively.
Thus, $\Gf$ and $\Gh$ are finite-difference estimates of the Jacobian--vector products
$J_h(x_t)\,\wnabla f(x_t)$ and $J_h(x_t)\,\wJ_h(x_t)^\top$. Then, we can set the Lagrangian multiplier $\lambda$ to be $\lambda_t = -(G_h)^{-1}(G_f - Kh(x_t))$, which leads to our zeroth-order feedback linearization algorithm (ZOFL). 
The full procedure is summarized in Algorithm~\ref{alg:equality-zeroth-order}.
\begin{algorithm}[htbp]
\caption{ZOFL (equality constraints)}
\label{alg:equality-zeroth-order}
\begin{algorithmic}[1]
\REQUIRE {Initial point $x_0$, algorithm hyperparameters: $T_G, T_B, r_1, r_2, K, \eta_t$}
\FOR{$t= 0,1,2,\dots, T_G$}
\STATE \textbf{Step 1}: Compute gradient estimation $\wnabla f(x_t), \widetilde{J}_h(x_t)$ using \cref{eq:two-point-estimator}.
\STATE \textbf{Step 2}: Given the gradient estimation $\wnabla f(x_t), \widetilde{J}_h(x_t)$, calculate $\lambda_t$ as follows
\begin{itemize}[left=0cm, itemsep=0pt, parsep=0pt, topsep=3pt, partopsep=0pt]
    \item Step 2.1: Compute $\Gf, \Gh$ that approximate $J_h(x_t)\wnabla f(x_t), J_h(x_t) \widetilde{J}_h(x_t)^\top $ as in \cref{eq:def-Gf-Gh}.
     
    \item Step 2.2: Set $\lambda_t = -\Gh^{-1}(\Gf-Kh(x_t))$
\end{itemize}
\STATE \textbf{Step 3}: Perform update $x_{t+1} = x_t - \eta_t\left(\wnabla f(x_t) + \wJ_h(x_t)^\top \lambda_t \right)$
\ENDFOR
\end{algorithmic}
\end{algorithm}
\subsection{Theoretical Guarantees on Constraint Satisfaction}
Building on the feedback–linearization perspective, the proposed ZOFL algorithm is designed to reduce constraint violations.
We now formalize this intuition by showing that, under mild regularity assumptions, the algorithm guarantees constraint satisfaction with high probability.

We begin by stating the assumptions on boundedness, smoothness, and conditioning that will be used throughout the analysis.
\begin{assump}[Bounded iterates]\label{assump:boundedness}
    The trajectory~$\{x_t\}$ of the algorithm lies inside a compact set~$\mathcal{D}\subset \bR^n$.
\end{assump}
\begin{assump}[Objective regularity]\label{assump:f-norms}
    The objective function~$f$ is differentiable on~$\mathcal{D}$ and satisfies
    \vspace{-10pt}
    \begin{align*}
        \|\nabla f(x)\| \le L_f, ~~\forall~x\in\mathcal{D}.
    \end{align*}
\end{assump}
\begin{assump}[Constraint regularity and conditioning]\label{assump:h-norms}
    The constraint function~$h$ is~$C^3$, i.e., three times continuously differentiable on~$\mathcal{D}$, and there exist constants $H,\overline L_h,\underline L_h,M,R>0$ such that for all $x\in\mathcal{D}$:
    \begin{align*}
       \|h(x)\|\!\le \!H,~~ \|J_h(x)\|\! \le\! \overline{L}_h, ~~\sigma_{\!\min}(J_h(x)) \!\ge\! \underline{L}_h,~~ \|D^2h(x)\| \!\le\! M, ~~\|D^3 h(x)\|_{\diag} \!\le\! R
    \end{align*}
    where $D^2h(x), D^3h(x)$ and $\|\cdot\|_{\diag}$ are defined as in Definition \ref{defi:directional-norm}.
\end{assump}
\begin{defi}[Second and Third-order directional derivative norm]\label{defi:directional-norm}
    Let \( f : \mathbb{R}^n \to \mathbb{R}^m \) be \( C^3 \). Then:

\begin{itemize}[itemsep=-10pt,topsep = 0pt,left=0pt]
    \item The second derivative \( D^2 f(x) \) is a symmetric bilinear map:
    \begin{talign*}
    D^2 f(x) : \mathbb{R}^n \times \mathbb{R}^n \to \mathbb{R}^m, \quad
    D^2 f(x)[u, v] := \left. \frac{\partial^2}{\partial s \partial t} f(x + s u + t v) \right|_{s=t=0}.
    \end{talign*}

    \item The third derivative \( D^3 f(x) \) is a symmetric trilinear map:
    \begin{talign*}
    D^3 f(x) : \mathbb{R}^n \!\times\! \mathbb{R}^n \!\times\! \mathbb{R}^n \!\to\! \mathbb{R}^m, ~
    D^3 f(x)[u, v, w] \!:=\! \left. \frac{\partial^3}{\partial s \partial t \partial r} f(x\! +\! s u \!+\! t v \!+\! r w) \right|_{s=t=r=0}\!\!.
    \end{talign*}
\end{itemize}
We define the diagonal norms of \( D^2 f(x) \) and \( D^3 f(x) \) as follows:
\begin{talign*}
\|D^2 f(x)\|_{\mathrm{diag}} \!:= \!\sup_{\|u\| \!=\! 1} \!\| D^2 f(x)[u, u] \|, ~ 
\|D^3 f(x)\|_{\mathrm{diag}} \!:=\! \sup_{\|u\| = 1} \| D^3 f(x)[u, u, u] \|.
\end{talign*}
\end{defi}

With these assumptions in place, we can formally state our main guarantee on constraint satisfaction.
\begin{theorem}\label{thm:zeroth-order-eq}
Suppose Assumptions~\ref{assump:boundedness}--\ref{assump:h-norms} hold and~$K\succ 0$.
Run \cref{alg:equality-zeroth-order} with $u_i$'s n \eqref{eq:two-point-estimator} drawn i.i.d. from the unit sphere.  
Fix~$\delta\in(0,1)$ and horizon~$T_G\in\mathbb N$. 
If the batch size~$T_B$ and probe radii~$r_1,r_2$ satisfy (cf. Appendix Lemma \ref{lemma:bound-lmin-JwJ})
\begin{talign*}
T_B \ge 32\!\left(m \log\!\Big(\frac{192\,n\,\overline L_h^2}{\underline L_h^2}\Big) + \log\!\Big(\frac{T_G}{\delta}\Big)\right),\qquad
r_1 \le \frac{\underline L_h}{8\sqrt{2\,\overline L_h\,R}}, 
\qquad
r_2 \le\frac{\underline L_h}{8\sqrt{2\,n\,\overline L_h\,R}},
\end{talign*}
and the stepsizes obey the stability condition~$0<\eta_t\,\lambda_{\min}(K)<1$ for all~$t$, then with probability at least~$1-\delta$, for all $t=1,\dots,T_G$,
    \begin{talign*}
\|h(x_t)]\|
\le \prod_{s=0}^{t-1}(1-\eta_s\lambda_{\min}(K))\|h(x_0)\| 
+ C_2 r_2^2 \sum_{s=0}^{t-1} \prod_{\tau=s+1}^{t-1}(1-\eta_\tau\lambda_{\min}(K))\,\eta_s  \\+ C_1 \sum_{s=0}^{t-1} \prod_{\tau=s+1}^{t-1}(1-\eta_\tau\lambda_{\min}(K))\,\eta_s^2,
\end{talign*}
where
\begin{talign*}
C_1 \;=\; M\!\left(nL_f + \frac{64\,n\,\overline L_h\, (nL_f\overline L_h + \|K\|\,H)}{\underline L_h^2} \right), 
\qquad
C_2 \;=\; nR\!\left(L_f + \frac{64\,\overline L_h\, (nL_f\overline L_h + \|K\|\,H)}{\underline L_h^2}\right).
\end{talign*}
In particular, for a constant step $\eta_t=\eta$,
\begin{equation}\label{eq:bound-constant-stepsize}
\textstyle \|h(x_t)\| 
\;\le\; (1-\eta\lambda_{\min}(K))^t\,\|h(x_0)\|
~+~ \frac{C_2 r_2^2}{\lambda_{\min}(K)}
~+~ \frac{C_1 \eta}{\lambda_{\min}(K)}.
\end{equation}
For a diminishing step $\eta_t=\eta/\sqrt{t}$,
\begin{equation}\label{eq:bound-diminishing-stepsize}
\textstyle\|h(x_t)\|
\;\le\; e^{-\eta\lambda_{\min}(K)(\sqrt{t}-1)}\,\|h(x_0)\|
~+~ \frac{2e\,C_2 r_2^2}{\lambda_{\min}(K)}
~+~ \frac{C_1 \eta\,e^{2-\eta\sqrt{t}}}{\lambda_{\min}(K)}
~+~ \frac{2e\,C_1\eta}{\lambda_{\min}(K)\sqrt{t+1}}.
\end{equation}
\end{theorem}

    

\begin{rmk}[Interpretation of Constraint Violation Bound]\label{rmk:interpretation} We now unpack the meaning of the bound in \cref{eq:bound-constant-stepsize}.
The first term, $(1-\eta\lmin(K))^t\|h(x_0)\|$, decays exponentially in~$t$ to zero.
This reflects the core effect of the FL design: in the absence of estimation or discretization errors, the constraint dynamics reduce to a simple stable linear system, driving violations to zero at a geometric rate.
In this sense, ZOFL inherits the strong feasibility guarantees of first-order FL.

The second term,~$\frac{C_2r_2^2}{\lmin(K)}\!\sim\! O(r_2^2) $, arises from replacing the exact Jacobian-vector products~$J_h(x_t)\nabla f(x_t)$ and~$J_h(x_t)J_h(x_t)^\top$ with their ZO approximations~$\Gf,\Gh$.
Because these approximations are based on finite-difference probing with radius~$r_2$, the residual scales quadratically in~$r_2$.
This error is fully controllable: if function evaluations of~$f,h$ are exact, one can make this term arbitrarily small by shrinking~$r_2$, up to the limits of numerical precision.
Thus, this term does not represent a fundamental barrier but rather a trade-off between accuracy and evaluation cost.

The third term,~$\frac{C_1\eta}{\lmin(K)}\sim O(\eta)$, comes from higher-order terms in the Taylor expansion of the constraint dynamics.
Unlike the approximation error, this residual is intrinsic to the Euler discretization used in ZOFL, where we approximate $h(x_{t+1}) - h(x_t)$ with the first order Taylor expansion $J_h(x_t)(x_{t+1} - x_t)$ (see \eqref{eq:h-dynamics-ZO}). Thus a fixed step size~$\eta$ produces a non-vanishing bias. This is the main bottleneck for achieving exact feasibility under constant step sizes.

To mitigate this discretization bias, one can use a diminishing schedule as in \cref{eq:bound-diminishing-stepsize}.
In this case, the residual terms vanish asymptotically, and constraint violations eventually disappear. 
The trade-off is that the ideal contraction term slows down: instead of exponential decay, the dominant term becomes~$e^{-\eta\lambda_{\min}(K)(\sqrt{t}-1)}\|h(x_0)\|$, which decreases subexponentially in $t$. 
This mirrors a common theme in stochastic optimization: stronger asymptotic guarantees are possible, but at the cost of slower transient progress.

In summary, the bound neatly separates three effects: (i) exponential contraction from FL, (ii) a controllable~$O(r_2^2)$ error from zeroth-order approximation, and (iii) an~$O(\eta)$ residual from discretization. 
Constant stepsizes yield fast initial reduction but leave a small feasibility gap, while diminishing stepsizes remove the gap but slow down the rate. 
This trade-off will guide the practical choice of stepsize and probing radius.

We also note that the batch size $T_B$ for the two-point estimator scales only with the number of constraints, $T_B \sim \widetilde{O}(m)$. Consequently, our algorithm is particularly efficient when the number of constraints is smaller than the number of variables, requiring only a small batch size at each iteration. 
\end{rmk}

\subsection{Global Convergence}\label{sec:global-convergence}

While Theorem \ref{thm:zeroth-order-eq} establishes guarantees for constraint satisfaction, it does not address the convergence to a KKT point. In this section, we will establish global convergence guarantees for the equality constrained setting \eqref{eq:optimization-eq-contraint}. We will make the following additional assumption
\begin{assump}\label{assump:f-norms-2}
 The objective function is $M_f$-smooth, i.e.,  $\|\nabla^2 f(x)\|\le M_f$ and its third order derivative satisfies $\|D^3f(x)\|_\diag\le R$.
\end{assump}
Further in order for the analysis to carry through, we add an additional line of code between Step 2.1 and 2.2 in Algorithm \ref{alg:equality-zeroth-order} to check if $\sigma_{\min}(G_h) \ge \sigma >0$, where $\sigma$ is some constant that we choose such that $\sigma > \frac{\underline L_h}{64}$, and if this condition is not met, we reject the sampled $u_i$'s and redo the sampling until the condition is met.
\begin{theorem}\label{thm:global-convergence}
    Under Assumption \ref{assump:boundedness}, \ref{assump:f-norms}, \ref{assump:h-norms}, and \ref{assump:f-norms-2}, by running the above described modified Algorithm \ref{alg:equality-zeroth-order}, for $\eta_t \!=\! \eta \!\le\! \frac{1}{6n(M_f+\tau M)}$ where $\tau \!=\! 64(T_BL_f+H)\overline{L}_h\underline{L}_h^{-2}\frac{\|K\|}{\lmin(K)}$, we have that
    \begin{talign*}
        \liminf_{t\to +\infty} \left( \|\nabla f(x_t) + J_h(x_t)^\top \lambda^\star(x_t)\|^2 + \|h(x_t)\|_1\right) \le {\frac{64 T_B\epsilon}{\eta n} }, \end{talign*}
    where $\lambda^\star(x_t) := (J_h(x_t)J_h(x_t)^\top)^{-1} J_h(x_t)\nabla f(x_t), \epsilon := n\overline{L}_h\underline{L}_h^{-2}Rr_1^2 + C_2r_2^2$ ($C_2$ defined as in Theorem \ref{thm:zeroth-order-eq}).
\end{theorem}
The proof of the theorem is in Appendix~\ref{apdx:global-convergence}. We would like to remark that the stepsize $\eta$ in the theorem is generally too small and the algorithm converges too slowly. Thus in practice, it is better to adopt linesearch methods to determine the stepsize adaptively. We would like to leave it as future work about more efficient linesearch design.
\section{Extension to Inequality-constrained Setting}\label{sec:ineq}

So far we have focused on equality constraints of the form as in \cref{eq:optimization-eq-contraint}.
We now consider the more general problem with inequality constraints:
\vspace{-3pt}
\begin{equation}\label{eq:optimization-ineq-contraint}
 \textstyle   \min_{x\in \bR^n}f(x) \quad \text{s.t. }h(x)\leq 0.
\end{equation}
The KKT conditions are
\begin{equation}\label{eq:KKT-ineq-constraint}
~~ -\nabla f(x) - J_h(x)^\top\lambda = 0,\quad 
    h(x) \le 0,\quad 
    \lambda \ge 0,\quad 
    \lambda^\top h(x) = 0.
\end{equation}

\noindent{\textbf{First-order FL algorithm.}} We can again view this as a control problem (Fig. \ref{fig:control-perspective-constrained-opt}), with dynamics
\begin{equation}\label{eq:ineq-constraint-dynamic}
    \begin{split}
        \textstyle x_{t+1} - x_t = -\eta_t\left(\nabla f(x_t) + J_h(x_t)^\top \lambda_t\right), \quad 
    y_t = h(x_t),\quad 
    \lambda_t \ge 0.
    \end{split}
\end{equation}

Compared with the equality case, the difficulty lies in enforcing the non-negativity of multipliers and the complementary slackness condition~$\lambda^\top h(x)=0$. 
In~\citep{zhang2025constrainedoptimizationcontrolperspective}, this is achieved by designing a more intricate FL controller:

\begin{tcolorbox}[title={\small FO-FL for Inequality-Constrained Optimization}]
\vspace{-10pt}
\begin{equation}\label{eq:feedback-linearization-ineq}
\begin{split}
    &\textstyle x_{t+1} - x_t = -\eta_t\!\left(\nabla f(x_t) + J_h(x_t)^\top \lambda_t\right),\\
    & \textstyle \lambda_t = \argmin_{\lambda \ge 0}\Big\{
        \tfrac{1}{2}\lambda^\top J_h(x_t)J_h(x_t)^\top \lambda 
        \!+\! \lambda^\top\!\big(J_h(x_t)\nabla f(x_t) \!-\! K h(x_t)\big)
    \!\Big\}.
\end{split}
\end{equation}
\end{tcolorbox}

Unlike the equality-constrained case in \cref{eq:feedback-linearization}, where $\lambda_t$ admits a closed-form expression, here~$\lambda_t$ is defined implicitly through a quadratic program.
This introduces nonsmooth trajectories and complicates the extension to ZO settings.

\noindent\textbf{Naive zeroth-order attempt.}~
In the ZO regime, the dynamics become
\begin{equation*}
    x_{t+1} - x_t = -\eta_t\big(\wnabla f(x_t) + \wJ_h(x_t)^\top \lambda_t\big).
\end{equation*}
A natural extension of \cref{eq:feedback-linearization-ineq} is to replace gradients with their estimates, which gives:
\begin{align}\label{eq:lambda-intuitive}
  \textstyle  \lambda_t = \argmin_{\lambda \ge 0}\Big\{
        \tfrac{1}{2}\lambda^\top J_h(x_t)\wJ_h(x_t)^\top \lambda 
        + \lambda^\top\!\big(J_h(x_t)\wnabla f(x_t) - K h(x_t)\big)
    \Big\}.
\end{align}
However, this quadratic form is not guaranteed to be symmetric positive definite (since~$J_h(x_t)\wJ_h(x_t)^\top$ need not be symmetric), and the resulting optimization problem may be ill-posed.

\noindent\textbf{Refined derivation.}~
The key is to return to the KKT conditions of \cref{eq:feedback-linearization-ineq}.
For the exact (first-order) case,~$\lambda_t$ and an auxiliary slack variable~$s$ must satisfy
\begin{equation*}
    J_h(x_t)J_h(x_t)^\top \lambda_t + J_h(x_t)\nabla f(x_t) = K h(x_t) + s, 
\quad s^\top \lambda_t = 0, \quad s \ge 0, \quad \lambda_t \ge 0.
\end{equation*}
In the zeroth-order regime, we mirror this structure but replace exact terms with their estimators:
\begin{align}\label{eq:lambda-t-intuition-KKT}
    J_h(x_t)\wJ_h(x_t)^\top \lambda_t + J_h(x_t)\wnabla f(x_t) = K h(x_t) + s, 
    ~~ s^\top \lambda_t = 0, ~~ s \ge 0, ~~ \lambda_t \ge 0.
\end{align}
This system defines~$\lambda_t$ without requiring~$J_h(x_t)\wJ_h(x_t)^\top$ to be symmetric positive definite. 
Our analysis confirms that \cref{eq:lambda-t-intuition-KKT} provides the correct formulation for ensuring feasibility.

Moreover,  as in the equality-constrained case, full Jacobians are not required.
It suffices to estimate the products
\vspace{-5pt}
\begin{equation*}
\textstyle G_f \approx J_h(x_t)\wnabla f(x_t), 
\qquad 
G_h \approx J_h(x_t)\wJ_h(x_t)^\top,
\end{equation*}
which can be obtained from the two-point estimators in \cref{eq:def-Gf-Gh}.
The resulting ZOFL scheme for inequality constraints is summarized below.


\begin{algorithm}[htbp]
\caption{ZOFL (inequality constraints)}
\label{alg:inequality-zeroth-order}
\begin{algorithmic}[1]
\REQUIRE {Initial point $x_0$, algorithm hyperparameters: $T_G, T_B, r_1, r_2, K, \eta$}
\FOR{$t= 0,1,2,\dots, T_G$}
\STATE \textbf{Step 1}: Compute gradient estimation $\wnabla f(x_t), \widetilde{J}_h(x_t)$ using \cref{eq:two-point-estimator}.
\STATE \textbf{Step 2}: Given the gradient estimation $\wnabla f(x_t), \widetilde{J}_h(x_t)$, calculate $\lambda_t$ as follows
\begin{itemize}[left=0cm, itemsep=0pt, parsep=0pt, topsep=3pt, partopsep=0pt]
    \item Step 2.1: Compute $\Gf, \Gh$ that approximate $J_h(x_t)\wnabla f(x_t), J_h(x_t) \widetilde{J}_h(x_t)^\top $ as in \cref{eq:def-Gf-Gh}.
    \item Step 2.2: Solve the following equations:
\begin{equation*}
    \begin{split}
        \textstyle \Gh\lambda + \Gf = Kh(x_t) + s,~~~s^\top \lambda = 0,~~~s\ge 0, ~~~\lambda\ge0
    \end{split}
    \end{equation*}
    Set $\lambda_t$ to be the solution for $\lambda$.
\end{itemize}
\STATE \textbf{Step 3}: Perform update $x_{t+1} = x_t - \eta\left(\wnabla f(x_t) + \wJ_h(x_t)^\top \lambda_t \right)$
\ENDFOR
\end{algorithmic}
\end{algorithm}

Our following theoretical analysis further validates \cref{alg:inequality-zeroth-order}'s ability to guarantee constraint satisfaction.

\noindent\textbf{Theoretical guarantees.}~
We now state the main feasibility result.
The proof follows the same high-level structure as \cref{thm:zeroth-order-eq} but requires sharper bounds on the error terms due to the nonsmooth projection step.

\begin{theorem}[Feasibility with Inequality Constraints]\label{thm:zeroth-order-ineq}
    Under Assumption \ref{assump:boundedness}, \ref{assump:f-norms} and \ref{assump:h-norms}, suppose~$T_B$ and~$r_1,r_2$ are chosen as in \cref{thm:zeroth-order-eq}.
    Then with probability at least~$1-\delta$, the ZOFL algorithm for inequality constraints (\cref{alg:inequality-zeroth-order}) satisfies
    \begin{talign*}
\|[h(x_t)]_+\|
\le \prod_{s=0}^{t-1}(1\!-\!\eta_s\lambda_{\min}(K))\|[h(x_0)]_+\| 
+ C_2 r_2^2 \!\sum_{s=0}^{t-1} \!\prod_{\tau=s\!+\!1}^{t-1}\!(1\!-\!\eta_\tau\lambda_{\min}(K))\,\eta_s  \\+ C_1 \!\sum_{s=0}^{t-1}\! \prod_{\tau=s\!+\!1}^{t-1}\!(1\!-\!\eta_\tau\lambda_{\min}(K))\,\eta_s^2,
\end{talign*}
for all~$t=1,\dots,T_G$, where $[h(x)]_+=\max\{h(x),0\}$ denotes the positive part of the constraint. 
Here the constants are
\[\textstyle 
C_1 = M\!\left(nL_f + \tfrac{64n\overline{L}_h(nL_f\overline{L}_h+\|K\|H)}{\underline{L}_h^2}\right), 
~~C_2 = n^2\overline{L}_h^2R\!\left(\tfrac{4096 n \overline{L}_h(L_f\overline{L}_h+\|K\|H)}{\underline{L}_h^4} + \tfrac{64L_f}{\underline{L}_h^2}\right).
\]
In particular, for constant step $\eta_t=\eta$,
\[
\|[h(x_t)]_+\| \le (1-\eta\lambda_{\min}(K))^t\|[h(x_0)]_+\| + \tfrac{C_2 r_2^2}{\lambda_{\min}(K)} + \tfrac{C_1\eta}{\lambda_{\min}(K)}.
\]
For diminishing step $\eta_t=\eta/\sqrt{t}$, the bound improves asymptotically as in Theorem~\ref{thm:zeroth-order-eq}.
\end{theorem}

The structure of the bound mirrors the equality-constrained case: exponential contraction toward feasibility, plus two residual terms accounting for zeroth-order approximation and discretization. 
The detailed interpretation in Remark \ref{rmk:interpretation} applies here as well, with the caveat that violations are measured via~$[h(x_t)]_+$ rather than~$h(x_t)$.
\section{Exploring midpoint methods for zeroth-order optimization}
\begin{algorithm}[htbp]
\caption{ZOFL-midpoint (equality constraints)}
\label{alg:equality-zeroth-order-midpoint}
\begin{algorithmic}[1]
\REQUIRE {Initial point $x_0$, algorithm hyperparameters: $T_G, T_B, r_1, r_2, K, \eta$}
\FOR{$t= 0,1,2,\dots, T_G$}
\STATE \textbf{Step 1}: Compute gradient estimation $\wnabla f(x_t), \widetilde{J}_h(x_t)$ using \cref{eq:two-point-estimator}.
\STATE \textbf{Step 2}: Given the gradient estimation $\wnabla f(x_t), \widetilde{J}_h(x_t)$, calculate $\lambda_t$ as follows
\begin{itemize}[left=0cm, itemsep=0pt, parsep=0pt, topsep=3pt, partopsep=0pt]
    \item Step 2.1: Compute $\Gf, \Gh$ that approximate $J_h(x_t)\wnabla f(x_t), J_h(x_t) \widetilde{J}_h(x_t)^\top $ as in \cref{eq:def-Gf-Gh}.
    \item Step 2.2: Set $\lambda_t = -\Gh^{-1}(\Gf-Kh(x_t))$
\end{itemize}
\STATE \textbf{Step 3}: Perform update $\xmid = x_t - \frac{\eta}{2}\left(\wnabla f(x_t) + \wJ_h(x_t)^\top \lambda_t \right)$
\STATE \textbf{Step 4}: Calculate $\wnabla f(\xmid), \widetilde{J}_h(\xmid)$ according to \cref{eq:two-point-estimator} (replace $x$ with $\xmid$) \emph{using the same $u_i$'s as in Step 1}
\STATE \textbf{Step 5}: 
\begin{itemize}[left=0cm, itemsep=0pt, parsep=0pt, topsep=3pt, partopsep=0pt]
    \item Step 5.1: Recalculate $\Gf, \Gh$ that approximate $J_h(\xmid)\wnabla f(\xmid), J_h(\xmid) \widetilde{J}_h(\xmid)^\top $ according to \cref{eq:def-Gf-Gh} (replace $x$ with $\xmid$)
    \item Step 5.2: Set $\lambda_t = -\Gh^{-1}(\Gf-Kh(x_t))$.
\end{itemize}
\STATE \textbf{Step 6}: Perform update $x_{t+1} = x_t - \frac{\eta}{2}\left(\wnabla f(\xmid) + \wJ_h(\xmid)^\top \lambda_t \right)$

\ENDFOR
\end{algorithmic}
\end{algorithm}
In Remark~\ref{rmk:interpretation}, we pointed out that discretization error is a major bottleneck in controlling constraint violation. This error arises from approximating $h(x_{t+1}) - h(x_t)$ using only the first-order term of the Taylor expansion, leading to an $O(\eta^2)$ residual. A natural question, then, is whether more accurate numerical schemes can reduce this error. Motivated by this, we introduce the midpoint method from numerical analysis (cf. \cite{suli2003introduction}), which achieves a discretization error of $O(\eta^3)$, and develop the midpoint variant of ZOFL (Algorithm~\ref{alg:equality-zeroth-order-midpoint}). Our experiments (Figures~\ref{fig:nonconvex-quadratic-programming} and \ref{fig:thermal-control}) demonstrate that this variant achieves improved constraint satisfaction compared to standard ZOFL. However, ZOFL-midpoint requires twice as many function evaluations per iteration, highlighting a trade-off between accuracy and sample efficiency. We further conjecture that the constraint violation bound under the midpoint method scales as $O(\eta^2)$, and leave a rigorous proof of this property as an open question.


\section{Numerical Validations}
We implement the ZOFL and ZOFL-midpoint algorithms (Algorithm \ref{alg:equality-zeroth-order} and \ref{alg:equality-zeroth-order-midpoint}) and compare it with the ZO-baseline method (\eqref{eq:ZO-baseline}) along with other baseline algorithms in zeroth-order constrained optimization, namely SZO-ConEx (\cite{nguyen_stochastic_2022}) and ZOGDA \cite{liu_min-max_2020}.

\noindent\textbf{Equality Constrained.} We consider the following nonconvex quadratic programming problem
\begin{talign*}
    &\min \frac{1}{2} x^\top x + c^\top x\qquad 
    s.t.~~ \frac{1}{2} x^\top x + a^\top x + b = 0,
\end{talign*}
where $x\in\bR^{100}$, $b=20$ and $a, c\in\bR^{100}$ are random vectors whose entry are sampled from a standard Gaussian distribution.

\noindent\textbf{Inequality Constrained.} We tested our algorithm on learning an efficient controller for building thermal regulation. We assume that the thermal dynamics to be a linear RC model (\cite{zhang2016decentralized,Li2022Distributed}) $x_{t+1} = Ax_t + Bu_t + d,$ where $x_t = \{x_{1,t}, x_{2,t},\dots, x_{n,t}\} \in \bR^n$ represents the temperature in each building at time step $t$, $u_t = \{u_{1,t}, u_{2,t},\dots, u_{n,t}\} $ is the thermal power injection and $d$ is the disturbances. We consider the controller ~$u_{i,t} = k_i x_{i,t} + b_i$~ and the optimization problem is given by optimizing the control parameters: $K = \{k_i\}_{i=1}^n, b = \{b_i\}_{i=1}^n$ to minimize the thermal energy subject to the thermal comfort constraint:
\begin{talign*}
    &\min_{K, b} \frac{1}{T}\sum_{t=0}^{T-1}\frac{1}{n}\sum_{i=1}^n u_{i,t}^2\\
    s.t. \quad &\frac{1}{T}\sum_{t=0}^{T-1}\frac{1}{n}\sum_{i=1}^n \max\left((x_{i,t}-x_{\textup{set}}),0\right)^2 - c\le 0\\
    & x_{t+1} = Ax_t + Bu_t + d, \qquad u_i = k_i x_{i,t} + b_i,
\end{talign*}
where we set $x_{\textup{set}}=22$\textsuperscript{o}C and $c = 1.5$. 

\begin{figure}[htbp]
    \centering
    \begin{subfigure}[b]{0.49\textwidth}
    \centering
        \includegraphics[width=\linewidth]{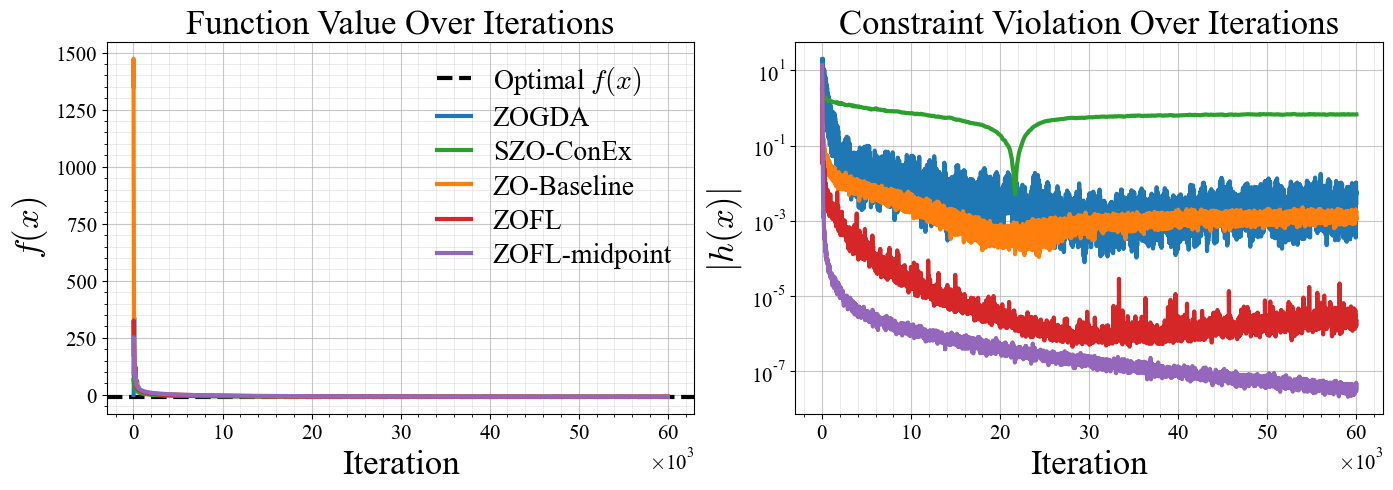}
    \vspace{-10pt}
    \subcaption{\centering \footnotesize Nonconvex quadratic programming\\ with Equality Constraints}
    \label{fig:nonconvex-quadratic-programming}
    \end{subfigure}
    \hfill
    \begin{subfigure}[b]{0.49\textwidth}
        \includegraphics[width=\linewidth]{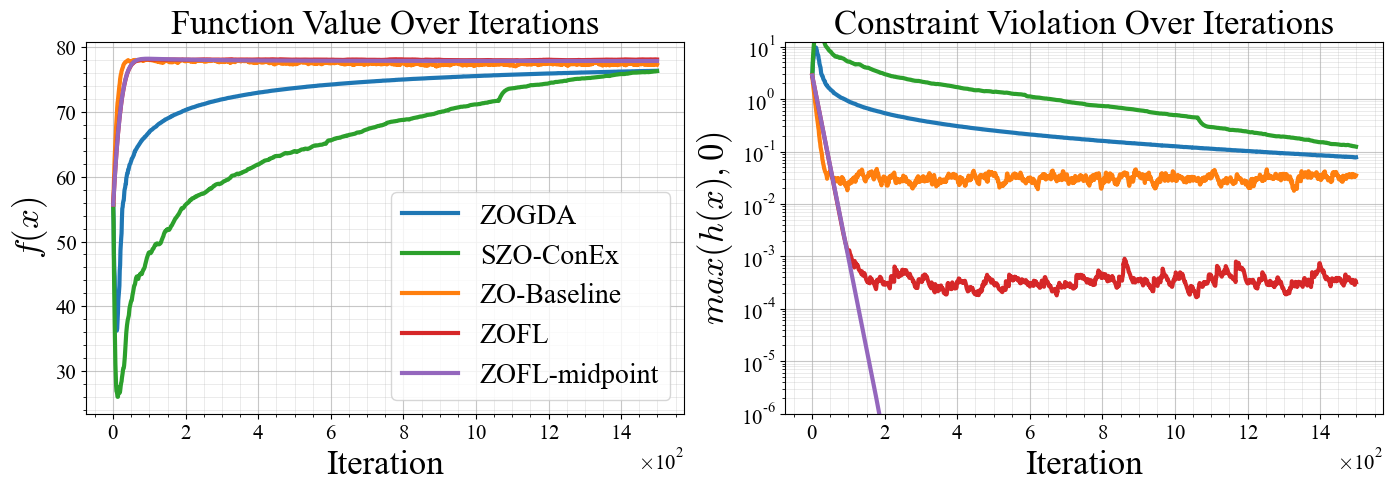}
    \vspace{-10pt}
    \subcaption{\centering\footnotesize Thermal Control with\\ Thermal Comfort Constraints}
    \label{fig:thermal-control}
    \end{subfigure}
\end{figure}
Figures \ref{fig:nonconvex-quadratic-programming} and \ref{fig:thermal-control} present the numerical results. Since diminishing step sizes often converge more slowly and are harder to tune in practice, we use a constant step size for the ZO algorithms. The left-hand plots show the cost function values, while the right-hand plots display the constraint violation. From the simulations, we observe that our algorithm, ZOFL, achieves better constraint satisfaction compared to the baseline methods while maintaining a similar cost. Moreover, ZOFL-midpoint further improves constraint satisfaction. These results suggest that, in safety-critical systems where constraint violations can have severe consequences, our algorithms are more favorable as they maintains safer operations.

\section{Conclusions}
We introduced a control-theoretic framework for zeroth-order constrained optimization, extending feedback linearization ideas to the derivative-free setting. 
Building on this perspective, we developed zeroth-order feedback linearization (ZOFL) algorithms that provide rigorous feasibility guarantees for both equality and inequality constraints, and we proposed a midpoint discretization variant that further reduces violation.
Our analysis shows that the FL perspective yields exponential contraction of constraint errors, while experiments confirm that ZOFL consistently achieves stronger feasibility with competitive objective values compared to existing baselines.

Despite these contributions, several limitations remain. 
Our guarantees rely on access to reasonably accurate zeroth-order oracles, and their robustness under biased or highly noisy evaluations is not yet established.
Moreover, although we prove finite-time bounds on constraint satisfaction and demonstrate strong empirical behavior, formal convergence to stationary points of the underlying problem remains open.
Addressing these challenges, through robust extensions, convergence analysis, and deployment in safety-critical domains, defines a promising direction for future work.

\section*{Acknowledgment}

The authors acknowledge the use of AI-assisted tools for improving grammar, clarity, and overall writing quality.

\appendix
\section{Other Related Works on Derivative Free Constrained Optimization}
Beyond zeroth-order (ZO) methods, several other lines of research in derivative-free constrained optimization have been developed.

One classical family of approaches is filter methods~\cite{audet_pattern_2004,echebest_inexact_2017,pourmohamad_statistical_2020,audet_progressive_2009,dzahini_constrained_2023,Luizzi10Sequential}, which are based on pattern search techniques. These methods iteratively reduce the objective function while attempting to decrease constraint violations, often through a progressive barrier function. While conceptually simple, filter methods generally rely on user-specified surrogate functions to generate candidate points and frequently require solving auxiliary subproblems. As a result, they are not easily generalizable to high-dimensional settings.

Another important class of derivative-free optimization techniques is model-based methods~\cite{muller_gosac_2017,augustin_nowpac_2015,gramacy_modeling_2016}, which build local surrogate models of the objective and constraints and optimize them iteratively. Such methods can achieve strong performance in low- to medium-dimensional problems, but their reliance on accurate surrogate models makes them sample-intensive and thus less practical in high-dimensional scenarios.

A different perspective is offered by extremum seeking (ES)~\cite{ariyur2003real}, which estimates gradients through deterministic perturbations of the system, typically sinusoidal probing signals, as opposed to random perturbations used in two-point estimators. The estimated gradient is then used to drive the system along a descent flow towards an extremum. ES shares a close connection with zeroth-order optimization: it can be interpreted as the continuous-time counterpart of single-point ZO methods~\cite{chen2022improve}. Recent works have begun to extend ES to constrained optimization settings~\cite{hazeleger_sampled-data_2022,chen_model-free_2022}, though its relationship to ZO approaches in this regime remains an open direction for future study.

\!\!Finally, Bayesian optimization (BO) represents another major branch of derivative-free optimization (see, e.g., \cite{gardner2014bayesian,hernandez2016general,acerbi2017practical}). BO adopts a fundamentally different philosophy: it constructs global probabilistic surrogate models (typically Gaussian processes) and leverages acquisition functions to trade off exploration and exploitation. BO is particularly well-suited for low- to medium-dimensional problems where function evaluations are costly, whereas ZO methods are more appropriate in high-dimensional regimes with relatively inexpensive evaluations.

\section{Proof of Theorem \ref{thm:zeroth-order-eq}}
\begin{proof}[Proof of Theorem \ref{thm:zeroth-order-eq}]
    From Taylor's expansion and Assumption \ref{assump:h-norms} we know that
    \begin{align*}
        y_{t+1} - y_t = J_h(x_t)(x_{t+1} - x_t) + \epsilon_t,
    \end{align*}
    where $$\textstyle \|\epsilon_t\| \le M\|x_{t+1}-x_t\|^2 \underset{\textup{Lemma \ref{lemma:auxiliary-stepsize}}}{\le} \underbrace{M\left(nL_f + \frac{64n\overline{L}_h (nL_f\overline{L}_h + \|K\|H)}{\underline{L}_h^2} \right)}_{:=C_1}\eta_t^2.$$
    We define an auxiliary variable $\lambda_t^\star \!:= \! - \left(J_h(x_t)\wJ_h(x_t)^\top\right)^{-1}\left(J_h(x_t)\wnabla f(x_t) \!-\! Kh(x_t)\right)$, and hence
    \begin{talign*}
        &\quad y_{t+1} - y_t = J_h(x_t)(x_{t+1} - x_t) + \epsilon_t\\
        &= -\eta_t J_h(x_t) \left(\wnabla f(x_t) + \wJ_h(x_t)^\top \lambda_t\right) + \epsilon_t\\
        &= -\eta_t J_h(x_t) \left(\wnabla f(x_t) + \wJ_h(x_t)^\top \lambda_t^\star\right)  + \eta_t J_h(x_t)\wJ_h(x_t)^\top (\lambda_t^\star - \lambda_t) + \epsilon_t \\
        &=-\eta_t \left(J_h(x_t)\wnabla f(x_t) - J_h(x_t)\wJ(x_t)^\top \left(J_h(x_t)\wJ(x_t)^\top \right)^{-1}(\wnabla f(x_t) + Kh(x_t))\right)\\
        &\qquad + \eta_t \underbrace{J_h(x_t)\wJ_h(x_t)^\top (\lambda_t^\star - \lambda_t)}_{:=\Delta_t} + \epsilon_t \\
        & = -\eta_t Ky_t + \eta_t\Delta_t + \epsilon_t\\
\Longrightarrow~~ &\quad \|y_{t+1}\| \le (1-\eta_t \lmin(K))\|y_t\| + \eta_t\|\Delta_t\| + C_1\eta_t^2
    \end{talign*}
    Further, from Lemma \ref{lemma:lambda-error}, we have that
    \begin{align*}
       \textstyle \|\Delta_t\| \le C_2r_2^2, ~~\textup{where } C_2 = nR\left(L_f + \frac{64\overline{L}_h(nL_f\overline{L}_h + \|K\|H)}{\underline{L}_h^2}\right)
    \end{align*}
Thus we get that
\begin{talign*}
    \|y_{t+1}\| &\le (1-\eta_t\lmin(K))\|y_t\| + \eta_t C_2 r_2^2 + C_1\eta_t^2\\
    \Longrightarrow \|y_t\|&\le \prod_{s=0}^{t-1} (1-\eta_s\lmin(K)) \|y_0\| + C_2r_2^2\sum_{s=0}^{t-1} \prod_{\tau=s+1}^{t-1}(1-\eta_\tau\lmin(K))\eta_s \\
    &\qquad+ C_1\sum_{s=0}^{t-1} \prod_{\tau=s+1}^{t-1}(1-\eta_\tau\lmin(K))\eta_s^2
\end{talign*}
In particular, if $\eta_t$ is set to a constant $\eta_t = \eta$, we have 
\begin{talign*}
    \|y_t\| &\!\le\! (1\!-\!\eta\lmin(K))^t\|y_0\| \!+\! C_2r_2^2 \sum_{s=0}^{t-1} (1\!-\!\eta\lmin(K))^s\eta \!+\! C_1 \eta^2\sum_{s=0}^{t-1}(1\!-\!\eta\lmin(K))^2\\
    & \le (1-\eta\lmin(K))^t\|y_0\| + \frac{C_2r_2^2}{\lmin(K)} + \frac{C_1\eta}{\lmin(K)}
\end{talign*}

If $\eta_t = \frac{\eta}{\sqrt{t+1}}$, then we have that
\begin{talign*}
    \|y_t\| &\le \prod_{s=0}^{t-1} \left(1\!-\!\frac{\eta\lmin(K)}{\sqrt{s+1}}\right) \|y_0\| \!+\! C_2r_2^2\sum_{s=0}^{t-1} \prod_{\tau=s+1}^{t-1}\left(1\!-\!\frac{\eta\lmin(K)}{\sqrt{\tau+1}}\right)\frac{\eta}{\sqrt{s+1}} \\&\qquad+ C_1\sum_{s=0}^{t-1} \prod_{\tau=s+1}^{t-1}\left(1-\frac{\eta\lmin(K)}{\sqrt{\tau+1}}\right)\frac{\eta^2}{s}\\
    &\underset{\textup{Lemma \ref{lemma:auxiliary-1}}}{\le} e^{-\eta\lmin(K)(\sqrt{t}-1)}\|y_0\| + C_2 r_2^2\eta e^{-\eta\sqrt{t}}\sum_{s=1}^{t}e^{\eta\lmin(K)\sqrt{s}}\frac{1}{\sqrt{s}} \\
    &\qquad\qquad\qquad\quad+ C_1\eta^2e^{-\eta\lmin(K)\sqrt{t}}\sum_{s=1}^{t}e^{\eta\lmin(K)\sqrt{s}}\frac{1}{s}\\
    &\!\!\!\!\underset{\textup{Lemma \ref{lemma:auxiliary-2}, \ref{lemma:auxiliary-3}}}{\le}\!\!\!\!\!\!\!\!\!\!\!\! e^{-\eta\lmin(K)(\sqrt{t}-1)}\|y_0\| \!+\! \frac{2eC_2 r_2^2}{\lmin(K)} \!+\! \frac{C_1\eta e^{2-\eta\sqrt{t}}}{\lmin(K)} \!+\! \frac{2eC_1\eta}{\lmin(K)\sqrt{t+1}}
\end{talign*}
\end{proof}
\subsection{Bounding $\|x_{t+1} - x_t\|$}
\begin{lemma}\label{lemma:auxiliary-stepsize}
    In Algorithm \ref{alg:equality-zeroth-order} we have that given
    \begin{talign*}
        T_B \ge 32\left(m \log\left(\frac{192\cdot n\cdot\overline{L}_h^2}{\underline{L}_h^2}\right) + \log\left(\frac{T_G}{\delta}\right)\right)\sim O\left(m\left(\log(n) + \log\left(\frac{\overline{L}_h}{\underline{L}_h}\right)\right) + \log\left(\frac{T_G}{\delta}\right)\right)
    \end{talign*}
    and $r_1 \le \frac{L_h}{8\sqrt{2\overline{L}_hR}} , \quad r_2 \le \frac{L_h}{8\sqrt{2n\overline{L}_hR}}$
    then with probability at least $1-\delta$,
    \begin{talign*}
        \|x_{t+1} - x_t\|\le \eta_t \left(nL_f + \frac{64n\overline{L}_h (L_f\overline{L}_h + \|K\|H)}{\underline{L}_h^2} \right)
    \end{talign*}
    holds for all $t=1,2,\dots, T_G$
    \begin{proof}
    
    From Assumption \ref{assump:f-norms} and the Cauchy mean value theorem
    \begin{align*}
        &|f(x_t+r_1u_i) - f(x_t-r_1u_i)|\le 2r_1\nabla f(x_t + \tilde r_1u_i)^\top u_i \le 2r_1L_f,\\
        \Longrightarrow~~& \|\wnabla f(x_t)\|\le nL_f.
    \end{align*}
    Similarly, from Cauchy mean value inequality we have that
    \begin{align*}
       \|\wJ(x_t)\| \le n\overline{L}_h, ~~\|G_f\| \le nL_f \overline{L}_h.
    \end{align*}
    Further, from Lemma \ref{lemma:bound-lmin-JwJ}, when $r_1 \le \frac{L_h}{8\sqrt{2\overline{L}_hR}} , r_2 \le \frac{L_h}{8\sqrt{2n\overline{L}_hR}}, $
    we have that \\$\sigma_{\min}(G_h)\ge \frac{\underline{L}_h^2}{64}.$
    Thus
    \begin{talign*}
        \|\lambda_t\| = \|G_h^{-1}(G_f - Kh(x_t))\| \le \frac{64 (L_f\overline{L}_h + \|K\|H)}{\underline{L}_h^2} 
    \end{talign*}
    Finally
    \begin{talign*}
        \|x_{t+1} - x_t\| &\le \eta_t (\|\wnabla f(x)\| + \|\wJ_h(x_t)\|\|\lambda_t\|)\le \eta_t \left(nL_f + \frac{64n\overline{L}_h (nL_f\overline{L}_h + \|K\|H)}{\underline{L}_h^2} \right),
    \end{talign*}
    which completes the proof.
    \end{proof}
\end{lemma}

\begin{lemma}\label{lemma:lambda-error}
    We define an auxiliary variable \\$\lambda_t^\star :=  - \left(J_h(x_t)\wJ_h(x_t)^\top\right)^{-1}\left(J_h(x)\wnabla f(x) - Kh(x)\right)$. Under the conditions as stated in Lemma \ref{lemma:auxiliary-stepsize}, we have that 
    \begin{talign*}
        \|J_h(x_t)\wJ_h(x_t)^\top (\lambda_t^\star - \lambda_t)\|\le nR\left(L_f + \frac{64\overline{L}_h(L_f\overline{L}_h + \|K\|H)}{\underline{L}_h^2}\right) r_2^2
    \end{talign*}
    \vspace{-10pt}
    \begin{proof}
        Further from Lemma \ref{lemma:bound-lmin-JwJ} we have $\|G_h^{-1}\|\le \frac{64}{\underline{L}_h^2}.$
        And thus
        \begin{talign*}
            &\quad \|J_h(x_t)\wJ_h(x_t)^\top (\lambda_t^\star - \lambda_t)\|\\
            &= \|J_h(x_t)\wnabla f(x_t) - Kh(x_t) + J_h(x_t)\wJ_h(x_t)^\top G_h^{-1}(G_f - Kh(x_t))\|\\
            & = \|J_h(x_t)\wnabla f(x_t) - G_f + (G_h - J_h(x_t)\wJ_h(x_t)^\top) G_h^{-1} (G_f - Kh(x))\|\\
            & \le \|J_h(x_t)\wnabla f(x_t) - G_f\| + \|G_h - J_h(x_t)\wJ_h(x_t)^\top\|\|G_h^{-1}\| (\|G_f\|+\|K\|H)\\
            &\stackrel{\textup{Lemma }\ref{lemma:taylor-error}}{\le} n {L}_f Rr_2^2 + n\overline{L}_h Rr_2^2 \frac{64}{\underline{L}_h^2} (L_f\overline{L}_h + \|K\|H)\\
            & = nR\left(L_f + \frac{64\overline{L}_h(nL_f\overline{L}_h + \|K\|H)}{\underline{L}_h^2}\right) r_2^2
        \end{talign*}
    \end{proof}
\end{lemma}

\section{Proof of Theorem \ref{thm:global-convergence}}\label{apdx:global-convergence}

\noindent\textbf{Notations. }Throughout this proof, we will use the notation $O(x)$ where $x$ is a positive scalar to denote any variable $\epsilon$ such that $\|\epsilon\|\le x$. Further, for $A\in\bR^{m\times n}$ that has full row rank, we use the notation $P_A:= A^\top (AA^\top)^{-1}A$ to denote the orthogonal projector onto the row space of $A$. For a matrix $A\in \bR^{m_1\times n}$ and $B\in \bR^{n\times m_2}$. We use $\sigma_{\min}(A|B):= \min_{v = Bu, u\in \bR^{m_2}} \frac{\|Av\|}{\|v\|}$ to denote the restricted minimum singular value of $A$ in the linear subspace of $B$.
\begin{proof}[Proof of Theorem \ref{thm:global-convergence}]
    At timestep $t$ in Algorithm \ref{alg:equality-zeroth-order}, let $U_t\!:= [u_1, u_2, \dots, u_{T_B}] \!\in\! \bR^{n\times T_B}$, then we have that
    \begin{talign*}
        \wnabla f(x_t) =  \frac{n}{T_B}\UU \nabla f(x_t) + O(nRr_1^2), ~~\wJ_h(x_t) = \frac{n}{T_B} J_h(x_t) \UU + O(nRr_1^2)
    \end{talign*}
    Further
    \begin{talign*}
        &\quad d_t = \wnabla f(x_t) + \wJ_h(x_t)^\top \lambda_t\\
        &\stackrel{\textup{Lemma\ref{lemma:lambda-error}}}{=} \wnabla f(x_t) - \wJ_h(x_t)^\top (J_h(x_t)\wJ_h(x_t)^\top)^{-1} (J_h(x_t)\wnabla f(x_t) - Kh(x_t)) + O(C_2r_2^2)\\
        & = \frac{n}{T_B}\UU \!\!\left(\nabla f(x_t) \!-\! J_h(x_t)^{\!\top}(J_h(x_t) \UU J_h(x_t)^{\!\top}\!)^{-1} (J_h(x_t)\UU\nabla f(x_t) \!-\! Kh(x_t))\right) \\
        &\quad +  \underbrace{O(n\overline{L}_h\underline{L}_h^{-2}Rr_1^2 + C_2r_2^2)}_{\epsilon}\\
        & = \frac{n}{T_B} \left(U_t (I\!-\!P_{J_h U_t}) U_t^\top \nabla f(x_t) +\UU J_h(x_t)^\top(J_h(x_t) \UU J_h(x_t)^\top)^{-1}Kh(x_t)\right)  \!+ \!\epsilon
    \end{talign*}
    Given this, we can verify that
    \begin{talign*}
        J_h(x_t) d_t &= \frac{n}{T_B}Kh(x_t) +\epsilon\\
        \nabla f(x_t)^\top d_t &= \frac{n}{T_B}\left(\|(I-P_{J_hU_t})U_t^\top \nabla f(x_t)\|^2 \right.\\&\left.+ \nabla f(x_t)^\top\UU J_h(x_t)^\top(J_h(x_t) \UU J_h(x_t)^\top)^{-1}Kh(x_t)\right) +\epsilon
    \end{talign*}
    From Assumption \ref{assump:h-norms}, we know that $h$ is $R$-smooth and thus
    \begin{talign*}
        &|h(x_{t+1}) - h(x_t) + \eta_tJ_h(x_t) d_t| \le \frac{M}{2} \|x_{t+1} - x_t\|^2\\
        \Longrightarrow &|h(x_{t+1}) - (I-\frac{n}{T_B}\eta_tK)h(x_t)| \le \frac{M}{2}\eta_t^2\|d_t\|^2 + \epsilon\\
        \Longrightarrow &\|h(x_{t+1})\|_1 - \|h(x_t)\|_1 \le -\eta_t\frac{n}{T_B}\lmin(K)\|h(x_t)\|_1 + \frac{M}{2}\eta_t^2 \|d_t\|^2 + \epsilon
    \end{talign*}
    Further, from the smoothness of $f$ we have that
    \begin{talign*}
        &\quad f(x_{t+1}) - f(x_t) \le -\eta_t\nabla f(x_t)^\top d_t + \frac{M_f}{2}\eta_t^2\|d_t\|^2 \\
        &\le \eta_t\frac{n}{T_B}(-\|(I-P_{J_hU_t})U_t^\top \nabla f(x_t)\|^2\\
        &\quad + \|\nabla f(x_t)^\top\UU J_h(x_t)^\top(J_h(x_t) \UU J_h(x_t)^\top)^{-1}Kh(x_t)\|)+ \epsilon\\
        &\le \eta_t\frac{n}{T_B}(-\|(I-P_{J_hU_t})U_t^\top \nabla f(x_t)\|^2+ 64T_BL_f\overline{L}_h\underline{L}_h^{-2}\|K\|\|h(x_t)\|_1) + \epsilon
    \end{talign*}
    Thus, by combining the above two inequalities, we have that for\\ $\tau \ge 64(T_BL_f+H)\overline{L}_h\underline{L}_h^{-2}\frac{\|K\|}{\lmin(K)}$ and $\phi(x):= f(x) +\tau \|h(x)\|_1$
    \begin{talign*}
    &\phi(x_{t+1}) - \phi(x_t)\le -\eta_t \frac{n}{T_B}(\|(I-P_{J_hU_t})U_t^\top \nabla f(x_t)\|^2 + 64H\overline{L}_h\underline{L}_h^{-2}\frac{\|K\|}{\lmin(K)}\|h(x_t)\|_1) \\&\qquad \qquad\qquad\qquad+ \frac{M_f+\tau M}{2}\eta_t^2\|d_t\|^2 + \epsilon
     \end{talign*}
     Further we have that
     \begin{talign*}
         \|d_t\|^2 \le 3\frac{n^2}{T_B}(\|(I-P_{J_hU_t})U_t^\top \nabla f(x_t)\|^2 + 64H\overline{L}_h\underline{L}_h^{-2}\frac{\|K\|}{\lmin(K)}\|h(x_t)\|_1 + \epsilon^2)
     \end{talign*}
     hence
     \begin{talign*}
       \phi(x_{t+1}) \!-\! \phi(x_t)  &\!\le\! (- \frac{\eta n_t}{T_B}\! + \!3\eta_t^2\frac{n^2(M_f\!+\!\tau M)}{T_B})(\|(I\!-\!P_{J_hU_t})U_t^\top \nabla f(x_t)\|^2 \!\!+ \!\|h(x_t)\|_1)
       \!+\!\epsilon
     \end{talign*}
     and thus for $\eta \le \frac{1}{6n(M_f+\tau M)}$ we have
     \begin{talign*}
         \phi(x_{t+1}) - \phi(x_t) \le -\frac{n}{2T_B}\eta (\|(I-P_{J_hU_t})U_t^\top \nabla f(x_t)\|^2 + \|h(x_t)\|_1) + \epsilon\\
         = -\frac{n}{2T_B}\eta (\|(I-P_{J_hU_t})U_t^\top (I-P_{J_h}) \nabla f(x_t)\|^2 + \|h(x_t)\|_1) + \epsilon
     \end{talign*}
     From the property of orthogonal projector $P_{J_h}, P_{J_h U_t}$ we have
     \begin{talign*}
         \|(I-P_{J_hU_t})U_t^\top (I-P_{J_h}) \nabla f(x_t)\| =\inf_{v\in col(J_h)} \|U_t^\top ((I-P_{J_h}) \nabla f(x_t)-v)\|\\
         \ge\sigma_{\min}(U_t^\top |[J_h(x_t)^{\!\top}\!\!,\nabla\! f(x_t)])\inf_{v\in col(J_h)}\| ((I-P_{J_h}) \nabla f(x_t)-v)\| \\= \sigma_{\min}(U_t^\top |[J_h(x_t)^{\!\top}\!\!,\nabla\! f(x_t)])\|(I-P_{J_h}) \nabla f(x_t)\|
     \end{talign*}
Thus by telescoping we get
\begin{talign*}
    \frac{1}{T}\!\sum_{t=1}^T\!\! \sigma_{\min}\!(U_t^\top |[J_h(x_t)^{\!\top}\!\!,\!\nabla\! f(x_t)])^2\|(I\!-\!P_{J_h}) \nabla f(x_t)\|^2 \!\!+\! \|h(x_t)\|_1 
    \!\le\! \frac{2T_B}{\eta n}\!\left(\!\frac{\phi(x_0)-\phi(T)}{T} \!+\! \epsilon\right)\\
    \Longrightarrow \!\liminf_{t\to+\infty} \left(\sigma_{\min}(U_t^\top |[J_h(x_t)^{\!\top}\!\!,\nabla\! f(x_t)])^2\|(I-P_{J_h}) \nabla f(x_t)\|^2 \!+\! \|h(x_t)\|_1 \right) \!\le \!\frac{2T_B\epsilon}{\eta n}
\end{talign*}
Further, given that the columns of $U_t$ is sampled i.i.d. from unit sphere, from Lemma \ref{lemma:auxi-bound-cov} we know that with a positive probability $(\sigma_{\min}(U_t^\top |[J_h(x_t)^{\!\top}\!\!,\nabla\! f(x_t)])^2 \ge \frac{1}{32}$, thus we have
\begin{talign*}
    \liminf_{t\to+\infty} \left(\|(I-P_{J_h}) \nabla f(x_t)\|^2 \!+\! \|h(x_t)\|_1 \right) \le \frac{64T_B\epsilon}{\eta n}.
\end{talign*}
      
      
\end{proof}
\section{Proof of Theorem \ref{thm:zeroth-order-ineq}}
\begin{proof}[Proof of Theorem \ref{thm:zeroth-order-ineq}]
    The proof follows a similar structure as the proof of Theorem \ref{thm:zeroth-order-eq}, with some substantial changes. From Taylor's expansion and Assumption \ref{assump:h-norms} we know that
    \vspace{-10pt}
    \begin{talign*}
        y_{t+1} - y_t = J_h(x_t)(x_{t+1} - x_t) + \epsilon_t,
    \end{talign*}
    where $\textstyle \|\epsilon_t\| \le M\|x_{t+1}-x_t\|^2 \underset{\textup{Lemma \ref{lemma:auxiliary-stepsize-ineq}}}{\le} \underbrace{M\left(nL_f + \frac{64n\overline{L}_h (L_f\overline{L}_h + \|K\|H)}{\underline{L}_h^2} \right)}_{:=C_1}\eta_t^2.$\\ 
    We define auxiliary variable $\lambda_t^\star,s^\star$ such that it satisfies the following sets of conditions
    \begin{align}\label{eq:def-lambda-star-ineq}
      \textstyle  \left(\!J_h(x_t)\wJ_h(x_t)^\top\!\right)\lambda_t^\star \!+ \!\left(J_h(x)\wnabla f(x) \!-\! Kh(x)\right) \!=\! s^\star, ~\lambda_t^\star\! \ge\! 0, ~s^\star \!\ge\! 0, ~(\lambda_t^\star)^{\!\top} s^\star\! =\!0.
    \end{align}
   and hence
    \begin{talign*}
        &\quad y_{t+1} - y_t = J_h(x_t)(x_{t+1} - x_t) + \epsilon_t= -\eta_t J_h(x_t) \left(\wnabla f(x_t) + \wJ_h(x_t)^\top \lambda_t\right) + \epsilon_t\\
        &= -\eta_t J_h(x_t) \left(\wnabla f(x_t) + \wJ_h(x_t)^\top \lambda_t^\star\right)  + \eta_t \underbrace{J_h(x_t)\wJ_h(x_t)^\top (\lambda_t^\star - \lambda_t)}_{:=\Delta_t} + \epsilon_t \\
        &= -\eta_t(Kh(x_t) + s^\star)+ \eta_t\Delta_t + \epsilon_t= -\eta_t Ky_t -\eta^t s^\star+ \eta_t\Delta_t + \epsilon_t
    \end{talign*}
    Since $s^\star \ge 0$, we have that
    \begin{talign*}
 \|[y_{t+1}]_+\| \le (1-\eta_t \lmin(K))\|[y_{t}]_+\| + \eta_t\|\Delta_t\| + C_1\eta_t^2
    \end{talign*}
    Further, from Lemma \ref{lemma:lambda-error-ineq}, we have that
    \begin{talign*}
        \|\Delta_t\| \le C_2r_2^2, ~~\textup{where } C_2 = n^2 \overline{L}_h^2R\left(\frac{4096 n \overline{L}_h(L_f\overline{L}_h + \|K\|H)}{\underline{L}_h^4} + \frac{64L_f}{\underline{L}_h^2}\right)
    \end{talign*}
Thus the rest of the proof follows exactly the same derivation as the proof of Theorem \ref{thm:zeroth-order-eq}, here we repeat as: 
\begin{talign*}
    \|[y_{t+1}]_+\| &\le (1-\eta_t\lmin(K))\|[y_{t}]_+\| + \eta_t C_2 r_2^2 + C_1\eta_t^2\\
    \Longrightarrow \|[y_{t}]_+\|&\le \prod_{s=0}^{t-1} (1-\eta_s\lmin(K)) \|[y_0]_+\| + C_2r_2^2\sum_{s=0}^{t-1} \prod_{\tau=s+1}^{t-1}(1-\eta_\tau\lmin(K))\eta_s \\
    &\qquad\qquad\qquad\qquad\qquad\qquad\qquad+ C_1\sum_{s=0}^{t-1} \prod_{\tau=s+1}^{t-1}(1-\eta_\tau\lmin(K))\eta_s^2
\end{talign*}
In particular, if $\eta_t$ is set to a constant $\eta_t = \eta$, we have 
\begin{talign*}
    \|[y_{t}]_+\| &\le (1-\eta\lmin(K))^t\|[y_0]_+\| + C_2r_2^2 \sum_{s=0}^{t-1} (1-\eta\lmin(K))^s\eta \\
    &\quad+ C_1 \eta^2\sum_{s=0}^{t-1}(1-\eta\lmin(K))^2\\
    & \le (1-\eta\lmin(K))^t\|[y_0]_+\| + \frac{C_2r_2^2}{\lmin(K)} + \frac{C_1\eta}{\lmin(K)}
\end{talign*}

If $\eta_t = \frac{\eta}{\sqrt{t+1}}$, then we have that
\begin{talign*}
    \|[y_{t}]_+\| &\le \prod_{s=0}^{t-1} \left(1-\frac{\eta\lmin(K)}{\sqrt{s+1}}\right) \|[y_0]_+\| + C_2r_2^2\sum_{s=0}^{t-1} \prod_{\tau=s+1}^{t-1}\left(1-\frac{\eta\lmin(K)}{\sqrt{\tau+1}}\right)\frac{\eta}{\sqrt{s+1}} \\&\qquad+ C_1\sum_{s=0}^{t-1} \prod_{\tau=s+1}^{t-1}\left(1-\frac{\eta\lmin(K)}{\sqrt{\tau+1}}\right)\frac{\eta^2}{s}\\
    &\underset{\textup{Lemma \ref{lemma:auxiliary-1}}}{\le} e^{-\eta\lmin(K)(\sqrt{t}-1)}\|[y_0]_+\| + C_2 r_2^2\eta e^{-\eta\sqrt{t}}\sum_{s=1}^{t}e^{\eta\lmin(K)\sqrt{s}}\frac{1}{\sqrt{s}} \\
    &\quad + C_1\eta^2e^{-\eta\lmin(K)\sqrt{t}}\sum_{s=1}^{t}e^{\eta\lmin(K)\sqrt{s}}\frac{1}{s}\\
    &\underset{\textup{Lemma \ref{lemma:auxiliary-2}, \ref{lemma:auxiliary-3}}}{\le} e^{-\eta\lmin(K)(\sqrt{t}-1)}\|[y_0]_+\| + \frac{2eC_2 r_2^2}{\lmin(K)} + \frac{C_1\eta e^{2-\eta\sqrt{t}}}{\lmin(K)} + \frac{2eC_1\eta}{\lmin(K)\sqrt{t+1}}
\end{talign*}
\end{proof}
\subsection{Bounding $\|x_{t+1} - x_t\|$}
\begin{lemma}\label{lemma:auxiliary-stepsize-ineq}
    In Algorithm \ref{alg:equality-zeroth-order} we have that given
    \begin{talign*}
        T_B \ge 32\left(m \log\left(\frac{192\cdot n\cdot\overline{L}_h^2}{\underline{L}_h^2}\right) + \log\left(\frac{T_G}{\delta}\right)\right)\sim O\left(m\left(\log(n) + \log\left(\frac{\overline{L}_h}{\underline{L}_h}\right)\right) + \log\left(\frac{T_G}{\delta}\right)\right)
    \end{talign*}
    and
        $r_1 \le \frac{L_h}{8\sqrt{2\overline{L}_hR}} , \quad r_2 \le \frac{L_h}{8\sqrt{2n\overline{L}_hR}},$    then with probability at least $1-\delta$,
    \begin{talign*}
        \|x_{t+1} - x_t\|\le \eta_t \left(nL_f + \frac{64n\overline{L}_h (L_f\overline{L}_h + \|K\|H)}{\underline{L}_h^2} \right)
    \end{talign*}
    holds for all $t=1,2,\dots, T_G$
    \begin{proof}
    
    From Assumption \ref{assump:f-norms} and the Cauchy mean value theorem
    \begin{talign*}
        &|f(x_t+r_1u_i) - f(x_t-r_1u_i)|\le 2r_1\nabla f(x_t + \tilde r_1u_i)^\top u_i \le 2r_1L_f,\\
        \Longrightarrow~~& \|\wnabla f(x_t)\|\le nL_f.
    \end{talign*}
    Similarly, from Cauchy mean value inequality we have that
    \begin{talign*}
       \|\wJ(x_t)\| \le n\overline{L}_h, ~~\|G_f\| \le L_f \overline{L}_h.
    \end{talign*}
    Further, from Lemma \ref{lemma:bound-lmin-JwJ}, when
        $r_1 \le \frac{L_h}{8\sqrt{2\overline{L}_hR}} , \quad r_2 \le \frac{L_h}{8\sqrt{2n\overline{L}_hR}},$
    we have that
    \begin{talign*}
        \sigma_{\min}(G_h)\ge \frac{\underline{L}_h^2}{64}.
    \end{talign*}
   Also, note that $\lambda_t$ is given by the following sets of equations
   \begin{talign*}
       G_h\lambda_t + G_f = Kh(x_t)  + s, ~~\lambda_t\ge0, ~~s\ge 0,~~ \lambda_t^\top s = 0.
   \end{talign*}
   Thus let the index set $\cI$ be $\cI:= \{i: s_i =0\}$, then we have that
   \begin{talign*}
       [\lambda_t]_{\cI^c} = 0,\quad
       [G_h]_{\cI\cI}[\lambda_t]_\cI + [G_h - Kh(x_t)]_\cI =0\\ 
   \Longrightarrow~
       \|\lambda_t\| = \left\|[G_h]^{-1}_{\cI\cI}[G_f - Kh(x_t)]_\cI\right\|
   \end{talign*}
   From Cauchy's Interlacing Theorem we get that
       $\sigma_{\min}([G_h]_{\cI\cI}) \ge \frac{\underline{L}_h^2}{64}$.
   Thus
   \begin{talign*}
       \|\lambda_t\| \le \frac{64}{\underline{L}_h^2} \|G_h - Kh(x_t)\|\le \frac{64n\overline{L}_h (L_f\overline{L}_h + \|K\|H)}{\underline{L}_h^2} 
   \end{talign*}
    \begin{talign*}
     \textup{Finally}~   \|x_{t+1} \!-\! x_t\| \le \eta_t (\|\wnabla f(x)\| \!+\! \|\wJ_h(x_t)\|\|\lambda_t\|)\le \eta_t \left(nL_f + \frac{64n\overline{L}_h (L_f\overline{L}_h + \|K\|H)}{\underline{L}_h^2} \right),
    \end{talign*}
    which completes the proof.
    \end{proof}
\end{lemma}

\begin{lemma}\label{lemma:lambda-error-ineq}
    We define the auxiliary variable $\lambda_t^\star$ as in \eqref{eq:def-lambda-star-ineq}. Under the conditions as stated in Lemma \ref{lemma:auxiliary-stepsize-ineq}, we have that 
    \begin{talign*}
        \|J_h(x_t)\wJ_h(x_t)^\top (\lambda_t^\star - \lambda_t)\|\le n^2 \overline{L}_h^2R\left(\frac{4096 n L_f\overline{L}_h^2}{\underline{L}_h^4} + \frac{64L_f}{\underline{L}_h^2}\right)r_2^2
    \end{talign*}
    \begin{proof}
    Define
    \begin{talign*}
       & ~~~A = J_h(x_t)\wJ_h(x_t)^\top, & b &= J_h(x_t)\wnabla f(x_t) - Kh(x)\\
        &\Delta A = G_h - J_h(x_t)\wJ_h(x_t)^\top,  &\hspace{-25pt}\Delta b &= G_f - J_h(x_t)\wnabla f(x_t)
    \end{talign*}
        From Lemma \ref{lemma:taylor-error}:
            $\|\Delta A\| \! \le\! n\overline{L}_h Rr_2^2, ~
            \|\Delta b\| \!\le\! nL_f Rr_2^2,~
            \| G_f\|, \!\|J_h(x_t)\wnabla f(x_t)\|\!\le\! nL_f\overline{L}_h.$

        For the sake of notational simplicity, in the proof we abbreviate $\lambda_t, \lambda_t^\star$ as $\lambda, \lambda^\star$. We define $A(\alpha), b(\alpha)$ as
        \begin{talign*}
            A(\alpha) := A + \alpha \Delta A,~~ B(\alpha) = B + \alpha \Delta B
        \end{talign*}
        and define $\lambda(\alpha)$ to be the solution of
        \begin{align}\label{eq:lambda-alpha-eq}
           \textstyle A(\alpha)\lambda(\alpha) + b(\alpha) = s(\alpha), ~~\lambda(\alpha)\ge 0, ~~s(\alpha)\ge 0,~~ \lambda(\alpha)^\top s(\alpha) = 0
        \end{align}
        Then it is clear that $\lambda = \lambda(1),\lambda^\star = \lambda(0)$.

        We can find a sequence of $\{\alpha_i\}$ such that $0=\alpha_0 < \alpha_1 < ... < \alpha_N = 1 $ such that with in each interval $\alpha, \alpha'\in (\alpha_i, \alpha_{i+1})$, $\lambda(\alpha)$ and $\lambda(\alpha')$ shares exactly the same support, which we denote as $\cI_i$. And from \eqref{eq:lambda-alpha-eq} we know that for any $\alpha\in [\alpha_1, \alpha_2]$, $\lambda(\alpha)$ can be written as follows:
        \begin{talign*}
            [\lambda(\alpha)]_{\cI_i} = [A(\alpha)]_{\cI_i\cI_i}^{-1}b(\alpha) ,~~
            [\lambda(\alpha)]_{\cI_i^c} = 0
        \end{talign*}
        And thus we get
        \begin{talign*}
            &\quad\|\lambda(\alpha_{i+1})-\lambda(\alpha_i)\|\le \|[A(\alpha_{i+1})]_{\cI_i\cI_i}^{-1}b(\alpha_{i+1}) - [A(\alpha_{i})]_{\cI_i\cI_i}^{-1}b(\alpha_{i})\|\\
            & \le\!\|\!A(\!\alpha_{i+1}\!)]_{\cI_i\cI_i}^{-1}\!\|\|A(\!\alpha_{i}\!)]_{\cI_i\cI_i}^{-1}\|A(\!\alpha_{i+1}\!)\!-\!A(\!\alpha_i\!)\|\|\|b(\!\alpha_{i+1}\!)\| \!+\! \|A(\alpha_{i})]_{\cI_i\cI_i}^{-1}\|\|b(\alpha_{i+1})\!-\!b(\alpha_i)\|\\
            &= (\alpha_{i+1}-\alpha_i)\left(\|A(\alpha_{i+1})]_{\cI_i\cI_i}^{-1}\|\|A(\alpha_{i})]_{\cI_i\cI_i}^{-1}\|\Delta A\|\|\|b(\alpha_{i+1})\| + \|A(\alpha_{i})]_{\cI_i\cI_i}^{-1}\|\|\Delta b\|\right)\\
            &\le (\alpha_{i+1}-\alpha_i)\left(\overline{L_h}\|A(\alpha_{i+1})]_{\cI_i\cI_i}^{-1}\|\|A(\alpha_{i})]_{\cI_i\cI_i}^{-1}\|\|b(\alpha_{i+1})\| + L_f\|A(\alpha_{i})]_{\cI_i\cI_i}^{-1}\|\right) nRr_2^2
        \end{talign*}
        Further, from Lemma \ref{lemma:bound-lmin-JwJ} and Cauchy's interlacing theorem (cf. \cite{hwang2004cauchy}) we know that for any principal minor of $A(\alpha)$ we have
        \begin{talign*}
            \|[A(\alpha)^{-1}]_{\cI\cI}\|\le \frac{64}{\underline{L}_h^2}, ~~\forall \alpha\in[0,1]
        \end{talign*}
        and clearly $\|b(\alpha)\|\le nL_f\overline{L}_h + \|K\|H$. And thus we get
        \begin{talign*}
            \|\lambda(\alpha_{i+1}) - \lambda(\alpha_i)\|\le (\alpha_{i+1}-\alpha_i) \left(\frac{4096 n \overline{L}_h(L_f\overline{L}_h + \|K\|H)}{\underline{L}_h^4} + \frac{64L_f}{\underline{L}_h^2}\right) nRr_2^2
        \end{talign*}
        And thus
        \begin{talign*}
            \|\lambda - \lambda^\star\| = \|\lambda(1) - \lambda(0)\| \le \left(\frac{4096 n \overline{L}_h(L_f\overline{L}_h + \|K\|H)}{\underline{L}_h^4} + \frac{64L_f}{\underline{L}_h^2}\right) nRr_2^2.
        \end{talign*}
        Thus
        \begin{talign*}
            &\|J_h(x_t)\wJ_h(x_t)^\top (\lambda_t^\star \!\!-\! \lambda_t)\|\le n\overline{L}_h^2 \|\lambda_t^\star \!-\! \lambda_t\| \!\le\! n^2 \overline{L}_h^2R\left(\frac{4096 n \overline{L}_h(L_f\overline{L}_h + \|K\|H)}{\underline{L}_h^4} \!+\! \frac{64L_f}{\underline{L}_h^2}\right)r_2^2
        \end{talign*}
    \end{proof}
\end{lemma}

\section{Bounding $\lmin(J_h(x_t)\wJ_h(x_t))$}
\begin{lemma}\label{lemma:bound-lmin-JwJ}
    In Algorithm \ref{alg:equality-zeroth-order}, we have that given
    \begin{talign*}
        T_B \ge 32\left(m \log\left(\frac{192\cdot n\cdot\overline{L}_h^2}{\underline{L}_h^2}\right) + \log\left(\frac{T_G}{\delta}\right)\right)\sim O\left(m\left(\log(n) + \log\left(\frac{\overline{L}_h}{\underline{L}_h}\right)\right) + \log\left(\frac{T_G}{\delta}\right)\right),
    \end{talign*}
    then with probability at least $1-\delta$
    \begin{talign*}
        \lmin(J_h(x_t)\wJ_h(x_t)^\top) \ge \frac{\underline{L}_h^2}{32} - \overline{L}_h Rr_1^2,~~
        \sigma_{\min}(G_h) \ge \frac{\underline{L}_h^2}{32} - \overline{L}_h R(r_1^2+nr_2^2).
    \end{talign*}
    for all $t=1, 2, \dots, T_G$
    \begin{proof}
        From Taylor's expansion and Assumption \ref{assump:h-norms} we have that
        \begin{talign*}
            \wJ_h(x_t) = \frac{n}{T_B}\sum_{i=1}^{T_b}J_h(x_t) u_i u_i^\top + \epsilon(x_t), 
        \end{talign*}
        where $\|\epsilon (x_t)\| \le Rr_1^2$.
        Thus
        \begin{talign*}
            J_h(x_t) \wJ_h(x_t)^\top = \frac{n}{T_B}\sum_{i=1}^{T_b}J_h(x_t) u_i u_i^\top J_h(x_t) + \widetilde{\epsilon}(x_t), ~~~\textup{where } \|\widetilde{\epsilon}(x_t)\|\le \overline{L}_h Rr_1^2
        \end{talign*}
        Further, from Lemma \ref{lemma:auxi-bound-cov} we have that when
        \begin{talign*}
            T_B \ge 32\left(m \log\left(\frac{192\cdot n\cdot\overline{L}_h^2}{\underline{L}_h^2}\right) + \log\left(\frac{T_G}{\delta}\right)\right)\sim O\left(m\left(\log(n) + \log\left(\frac{\overline{L}_h}{\underline{L}_h}\right)\right) + \log\left(\frac{T_G}{\delta}\right)\right)
        \end{talign*}
        then with probability at least $1-\delta$
        \begin{talign*}
            \lmin\left(\frac{n}{T_B}\sum_{i=1}^{T_B}J_h(x_t) u_i u_i^\top J_h(x_t)\right) \ge \frac{\underline{L}_h^2}{32}, ~~\forall ~ t=1,2,\dots, T_G
        \end{talign*}
        And thus
     \vspace{-20pt}
        \begin{talign*}
            \lmin(J_h(x_t)\wJ_h(x_t)^\top) \ge \frac{\underline{L}_h^2}{32} - \overline{L}_h Rr_1^2
        \end{talign*}
        Further, from Lemma \ref{lemma:taylor-error} we have $\|G_h - J_h(x_t)\wJ_h(x_t)^\top\| \le n\overline{L}_h Rr_2^2$
        and thus
        \begin{talign*}
            \sigma_{\min}(G_h) \ge \frac{\underline{L}_h^2}{32} - \overline{L}_h R(r_1^2+nr_2^2)
        \end{talign*}
        \vspace{-20pt}
    \end{proof}
\end{lemma}
Proving Lemma \ref{lemma:bound-lmin-JwJ} will need the following fundamental theorem that uses Small-ball condition to prove anti-concentration:
\begin{theorem}[Small-Ball Lower Bound on Minimum Eigenvalue on Empirical Covariance Matrix]\label{thm:small-ball}
Let $u_1, \dots, u_N \in \mathbb{R}^n$ be i.i.d. random vectors. Suppose there exist constants $\tau > 0$, $p > 0$, and $K > 0$ such that:
\vspace{5pt}
\begin{enumerate}
    \item (\textbf{Small-ball condition}) For all $z \in \mathbb{S}^{n-1}$:
$\mathbb{P}(|\langle u_i, z \rangle| \ge \tau) \ge p.$
\item For all $z \in \mathbb{S}^{n-1}$:
$|\langle u_i, z \rangle|^2 \le K.$
\end{enumerate}
\vspace{5pt}
Then for any $\delta \in (0, 1)$, if  $N \ge \frac{8}{p^2} \left( n \log\left(\frac{24K}{\tau^2 p}\right) + \log\left(\frac{1}{\delta}\right) \right)$
then with probability at least $1 - \delta$,
$\lambda_{\min}\left( \frac{1}{N} \sum_{i=1}^N u_i u_i^\top \right) \ge \frac{\tau^2 p}{4}.$
\begin{proof}
Define $S := \frac{1}{N} \sum_{i=1}^N u_i u_i^\top.$
Then for any $z \in \mathbb{S}^{n-1}$ we have $z^\top S z = \frac{1}{N} \sum_{i=1}^N \langle u_i, z \rangle^2$
Let $Z_i = \langle u_i, z \rangle^2$. By assumption, $\mathbb{P}(Z_i \ge \tau^2) \ge p$. Define indicator variables $A_i := \mathbb{I}\{Z_i \ge \tau^2\}$. Then $A_i \sim \text{Bernoulli}(\ge p)$, and
$z^\top S z \ge \frac{\tau^2}{N} \sum_{i=1}^N A_i.$ By the Chernoff bound, for all $N \ge 1$,
$\mathbb{P}\left( \sum_{i=1}^N A_i < \frac{pN}{2} \right) \le \exp\left(-\frac{p^2 N}{8}\right).$
Thus, with probability at least $1 - \exp\left(-\frac{p^2 N}{8}\right)$, for a fixed $z$,
$
z^\top S z \ge \frac{\tau^2 p}{2}.
$

Now construct an $\epsilon$-net $\mathcal{N}_\epsilon \subset \mathbb{S}^{n-1}$ with $|\mathcal{N}_\epsilon| \le (3/\epsilon)^n$. Using the union bound:
\begin{talign*}
\mathbb{P}\left( \exists z \in \mathcal{N}_\epsilon : z^\top S z < \frac{\tau^2 p}{2} \right) \le (3/\epsilon)^n \cdot \exp\left(-\frac{p^2 N}{8}\right).
\end{talign*}
To ensure the right hand side is $\le \delta$, it suffices that:
$N \ge \frac{8}{p^2} \left( n \log\left(\frac{3}{\epsilon}\right) + \log\left(\frac{1}{\delta}\right) \right).$
To extend from the net to all $z$, note:
\begin{talign*}
|z^\top S z - \hat{z}^\top S \hat{z}| = \left|(z+\hat{z}) S (z-\hat{z})\right|\le 2K\epsilon
\end{talign*}
And thus choosing $\epsilon \sim \frac{\tau^2 p}{8K}$ ensures for all $z \in \mathbb{S}^{n-1}$:
$z^\top S z \ge \frac{\tau^2 p}{2} -2 K \cdot \epsilon \ge \frac{\tau^2 p}{4}.$
This concludes the proof.
\end{proof}

\end{theorem}

The following lemma is an immediate corollary of Theorem \ref{thm:small-ball}.
\begin{lemma}\label{lemma:auxi-bound-cov}
    \!\!Given a fixed matrix $A\!\in\!\bR^{m\times n}$ and random variables $u_1, \!u_2,\dots,\!u_N \!\in\! \bR^n$ where $u_i$ is sampled i.i.d. from the unit sphere, we have that when
    \begin{talign*}
        N \!\ge\! 32\left(m \log\left(192\cdot n\cdot\kappa(AA^\top)\right) \!+ \!\log\left(\frac{1}{\delta}\right)\right)\sim O\left(m(\log(n) \!+ \!\log(\kappa(AA^\top)) \!+ \!\log\left(\frac{1}{\delta}\right)\right)\!.
    \end{talign*}
    where $\kappa(AA^\top):=\frac{\lmax(AA^\top)}{\lmin(AA^\top)}$, then with probability at least $1-\delta$,
    \begin{talign*}
        \lmin\left(\frac{n}{N}\sum_{i=1}^N Au_i  u_i^\top A^\top\right) \ge \frac{\lmin(AA^\top)}{32}
    \end{talign*}
    \vspace{-10pt}
    \begin{proof}
        Define $v_i = \sqrt{n}Au_i$. From the property of random uniform unit sphere vectors \cite{klartag2017super} we have that for any $z\in \bR^{n}$,
        \begin{talign*}
            \bP\left(|u_i^\top z| \ge \frac{1}{2\sqrt{n}} \|z\|\right) \ge 1/2
        \end{talign*}
        Thus for any $z'\in\bR^{m}$,
        \begin{talign*}
        \bP\left(|v_i^\top z'| \ge \frac{1}{2} \sgmin(A)\|z'\|\right)\ge
            \bP\left(|v_i^\top z'| \ge \frac{1}{2} \|Az'\|\right) = \bP\left(|u_i^\top z| \ge \frac{1}{2\sqrt{n}} \|z\|\right)\ge 1/2
        \end{talign*}
        Thus, the vector $v_i$'s satisfies Condition 1 in Theorem \ref{thm:small-ball} with
            $\tau = \frac{\sgmin(A)}{2}, \quad p = \frac{1}{2}.$
        Further, given that $u_i\in \bS^{n-1}$,
           $ |v_i^\top z'| = \sqrt{n}|u_i^\top A^\top z'|\le \sqrt{n}\|A\|.$
        Thus the vector $v_i$'s satisfies Condition 2 in Theorem \ref{thm:small-ball} with $K = n\|A\|^2.$
        Directly applying Theorem \ref{thm:small-ball} will finish the proof.
    \end{proof}
\end{lemma}

\section{Auxiliaries}
\begin{lemma}\label{lemma:taylor-error} For $\wJ(x)t), \wnabla f(x_t)$ defined in \eqref{eq:two-point-estimator} and $G_h, G_f$ in \eqref{eq:def-Gf-Gh}, we have
     \begin{align*}
&\textstyle \|G_h \!-\! J_h(x_t)\wJ_h(x_t)^\top\|\! \le\! n\overline{L}_h Rr_2^2, ~ \|G_f \!-\! J_h(x_t)\wnabla f(x_t)\|\!\le \!nL_f Rr_2^2, ~\|\wnabla f(x_t)\|\!\le\! nL_f. 
        \end{align*}
        \vspace{-10pt}
        \begin{proof}
            The first two inequalities obtained from standard truncation-error bounds for the central difference directional derivative \cite{suli2003introduction}. The last inequality can be derived by the Lagrange's Mean Value Theorem, where we can show that for any unit vector $u$, $\|\frac{f(x+ru)-f(x-ru)}{2r} u\| = \|u u^\top\nabla f(x+\bar ru)\|\le L_f $.
        \end{proof}
\end{lemma}
\begin{lemma}\label{lemma:auxiliary-1}
    For any $\eta \!>\!0, t,s \!\in \!\mathbb{N},  t \!\ge\! s\!\ge\!0$:
        $\prod_{\tau=s}^{t-1} (1-\frac{\eta}{\sqrt{\tau+1}}) \le e^{-\eta\left(\sqrt{t} - \sqrt{s+1}\right)} .$
    \begin{proof}
            $\quad \prod_{\tau=s}^{t-1} (1-\frac{\eta}{\sqrt{\tau+1}}) \le \prod_{\tau=s}^{t-1}e^{-\frac{\eta}{\sqrt{\tau+1}}}
            \le e^{-\eta\sum_{\tau=s}^{t-1}\frac{1}{\sqrt{\tau+1}}}\le e^{-\eta(\sqrt{t} - \sqrt{s+1})}$.
    \end{proof}
\end{lemma}
\begin{lemma}\label{lemma:auxiliary-2} For any $\eta >0$ and $t,s \in \mathbb{N},  t \ge s\ge0$: 
        $\sum_{s=1}^{t}e^{\eta\sqrt{s}}\frac{1}{\sqrt{s}} \le \frac{2}{\eta}2e^{\eta\sqrt{t+1}}$
    \begin{proof}
        $
            \sum_{s=1}^{t}e^{\eta\sqrt{s}}\frac{1}{\sqrt{s}} \le \int_{s=1}^{t+1} e^{\eta\sqrt{s}}\frac{
            1
            }{\sqrt{s}}ds = 2\int_{s=1}^{t+1} e^{\eta\sqrt{s}} d\sqrt{s} \le \frac{2}{\eta}e^{\eta\sqrt{t+1}} 
        $.
    \end{proof}
\end{lemma}
\begin{lemma}\label{lemma:auxiliary-3} For any $\eta >0$ and $t,s \in \mathbb{N},  t \ge s\ge0$
    \begin{talign*}
        \sum_{s=1}^{t}e^{\eta\sqrt{s}}\frac{1}{s} \le \frac{e^2}{\eta} + \frac{2}{\eta}\frac{1}{\sqrt{t+1}} e^{\eta\sqrt{t+1}}
    \end{talign*}
    \begin{proof}
    \vspace{-30pt}
    \begin{talign*}
       &\quad  \sum_{s=1}^{t}e^{\eta\sqrt{s}}\frac{1}{s} \le \int_{s=1}^{t+1} e^{\eta\sqrt{s}}\frac{
            1
            }{{s}}ds \le  2\int_{s=1}^{t+1} e^{\eta\sqrt{s}}\frac{1}{\sqrt{s}} d\sqrt{s}\\
            &= \int_{x=1}^{\sqrt{t+1}} \frac{1}{x}e^{\eta x}dx \le \int_{x=1}^{\frac{2}{\eta}} e^{\eta x}dx + \int_{x=\frac{2}{\eta}}^{\sqrt{t+1}} \frac{1}{x}e^{\eta x}dx\\
            &\le \int_{x=1}^{\frac{2}{\eta}} e^{\eta x} dx+ \int_{x=\frac{2}{\eta}}^{\sqrt{t+1}} \left(\frac{2}{x} - \frac{2}{\eta x^2}\right)e^{\eta x} dx\le \frac{e^2}{\eta} + \frac{2}{\eta}\frac{1}{\sqrt{t+1}} e^{\eta\sqrt{t+1}}.
            \end{talign*}
            \vspace{-20pt}
    \end{proof}
\end{lemma}

\setlength{\bibsep}{0pt}
\renewcommand{\url}[1]{}
\providecommand{\doi}[1]{}

\providecommand{\urlprefix}{}
\providecommand{\doiprefix}{}

\providecommand{\urldateprefix}{}
\bibliographystyle{abbrvnat}
\bibliography{references,bib}
\end{document}